\documentclass[11pt]{amsart}

\IfFileExists{srcltx.sty}{\usepackage{srcltx}}

\numberwithin{equation}{section}

\usepackage{tikz}
\usetikzlibrary{arrows}

\usepackage{bbold}

\usepackage[latin1]{inputenc}
\usepackage{xspace,amssymb,amsfonts,euscript}
\usepackage{amsthm,amsmath}
\usepackage{palatino}
\usepackage{euscript}
\input xy \xyoption {all}

\usepackage{pdfsync}
\usepackage{graphicx}

\RequirePackage{color}
\definecolor{myred}{rgb}{0.75,0,0}
\definecolor{mygreen}{rgb}{0,0.5,0}
\definecolor{myblue}{rgb}{0,0,0.65}

\RequirePackage{ifpdf}
\ifpdf
  \IfFileExists{pdfsync.sty}{\RequirePackage{pdfsync}}{}
  \RequirePackage[pdftex,
   colorlinks = true,
   urlcolor = myblue, 
   citecolor = mygreen, 
   linkcolor = myred, 
   pagebackref,
   bookmarksopen=true]{hyperref}
\else
  \RequirePackage[hypertex]{hyperref}
\fi
\RequirePackage{ae, aecompl, aeguill} 

    \def\CM{{\mathbb{C}}}
    \def\DM{{\mathbb{D}}}

  \def\hg{{\mathfrak h}}

    \def\RM{{\mathbb{R}}}

    \def\ZM{{\mathbb{Z}}}

    \def\BC{{\mathcal{B}}}
    
    \def\DC{{\mathcal{D}}}

    \def\HC{{\mathcal{H}}}
    
    \def\JC{{\mathcal{J}}}

\def\a{\alpha}
\def\b{\beta}
\def\g{\gamma}

\def\e{\varepsilon}

\newcommand{\nc}{\newcommand} \newcommand{\renc}{\renewcommand}

\newcommand{\rdots}{\mathinner{ \mkern1mu\raise1pt\hbox{.}
    \mkern2mu\raise4pt\hbox{.}
    \mkern2mu\raise7pt\vbox{\kern7pt\hbox{.}}\mkern1mu}}

\def\ov{\overline}
\def\un{\underline}

\def\to{\rightarrow}

\def\longto{\longrightarrow}

\def\onto{\twoheadrightarrow}
\nc{\triright}{\stackrel{[1]}{\to}}
\nc{\longtriright}{\stackrel{[1]}{\longto}}
\newcommand{\sumset}{\stackrel{\scriptstyle{\oplus}}{\scriptstyle{\subset}}}

\nc{\ot}{\otimes}

\nc{\HotRR}{{}_R\mathcal{K}_R}
\nc{\HotR}{\mathcal{K}_R}
\nc{\excise}[1]{}
\nc{\defect}{\text{df}}
\nc{\h}[1]{\underline{H}_{#1}}

\nc{\ptau}{\tau}

\nc{\Ga}{\mathbb{G}_a} 
\nc{\Gm}{\mathbb{G}_m} 

\nc{\Perv}{{\mathbf{P}}}

\nc{\IH}{{\mathrm{IH}}}

\nc{\ic}{\mathbf{IC}}

\nc{\gl}{{\mathfrak{gl}}}
\renc{\sl}{{\mathfrak{sl}}}
\renc{\sp}{{\mathfrak{sp}}}

\nc{\HBM}{H^{BM}}
\nc{\id}{\textrm{id}}

 \DeclareMathOperator{\Hom}{Hom}

 \DeclareMathOperator{\ch}{ch}

\newtheorem{thm}{Theorem}[section]
\newtheorem{lem}[thm]{Lemma}

\newtheorem{prop}[thm]{Proposition}
\newtheorem{cor}[thm]{Corollary}
  
\newtheorem{conj}[thm]{Conjecture}

\theoremstyle{definition}

\newtheorem{ex}[thm]{Example}

\theoremstyle{remark}
\newtheorem{remark}[thm]{Remark}

\def\bfa{\mathbf a}
\def\bfx{\mathbf x}

\def\gr{\textrm{gr}}

\nc{\simto}{\stackrel{\sim}{\to}}
\nc{\oBC}{\ov{\BC}} 

\nc{\pH}{H}
\nc{\pBC}{{}^p \BC}
\nc{\cell}{{\bf c}}

\begin{document}

\title [Relative hard Lefschetz for Soergel bimodules]{Relative hard Lefschetz for \\ Soergel bimodules}

\author{Ben Elias}
\address{University of Oregon, Eugene, Oregon, USA}
\email{belias@uoregon.edu}
\urladdr{}

\author{Geordie Williamson} 
\address{Max-Planck-Institut f\"ur Mathematik, Bonn, Germany}
\email{geordie@mpim-bonn.mpg.de}
\urladdr{}

\begin{abstract} We prove the relative hard Lefschetz
  theorem for Soergel bimodules. It follows that the structure
  constants of the Kazhdan-Lusztig basis are unimodal. We explain why 
  the relative hard Lefschetz theorem implies that
  the tensor category associated by Lusztig to any
  2-sided cell in a  Coxeter group is rigid and pivotal.
\end{abstract}

\maketitle

\tableofcontents

\section{Introduction} \label{sec:introduction}

Let $(W,S)$ denote a Coxeter system and $\HC$ its Hecke algebra. It is an
algebra over $\ZM[v^{\pm 1}]$ with
standard basis $\{ H_x \; | \; x \in W\}$ and Kazhdan-Lusztig basis $\{
\un{H}_x \; | \; x \in W \}$. The Kazhdan-Lusztig positivity
conjectures are the statements:
\begin{enumerate}
\item (``positivity of Kazhdan-Lusztig polynomials'') if we write $\un{H}_x = \sum h_{y,x} H_y$, then $h_{y,x} \in
  \ZM_{\ge 0}[v]$;
\item (``positivity of structure constants'') if we write $\un{H}_x\un{H}_y = \sum \mu_{x,y}^z \un{H}_z$ then
  $\mu_{x,y}^z \in \ZM_{\ge 0}[v^{\pm 1}]$.
\end{enumerate}
These conjectures have been known since the 1980s for Weyl groups
of Kac-Moody groups \cite{KaL2, Springer}, using sophisticated geometric technology. More recently in \cite{EWHodge} the authors proved
these conjectures algebraically for arbitrary Coxeter systems by
establishing Soergel's conjecture.
 
Let us briefly recall the setting of Soergel's conjecture.
For a certain reflection representation $\hg$ of $(W,S)$ over the real
numbers, Soergel constructed a category $\BC$ of Soergel bimodules,
which is a full subcategory of the category of graded $R$-bimodules,
where $R$ denotes the polynomial functions on $\hg$. The category of
Soergel bimodules $\BC$ is monoidal under tensor product of
bimodules, and is closed under grading shift.
Soergel showed that one has a canonical isomorphism
\[
\ch : [\BC] \simto \HC
\]
of $\ZM[v^{\pm 1}]$-algebras between the split Grothendieck group of Soergel bimodules and the
Hecke algebra. (The split Grothendieck group $[\BC]$ is an algebra via
$[B][B'] := [B \otimes_R B']$ and is a $\ZM[v^{\pm 1}]$-algebra via
$v[B] := [B(1)]$, where $(1)$ denotes a grading shift.) In proving this isomorphism, Soergel constructed
certain bimodules $B_x$ for each $x \in W$ which give representatives for all indecomposable Soergel
bimodules (up to isomorphism
and grading shift). Soergel's conjecture is the statement that
$\ch([B_x]) = \un{H}_x$, which immediately implies the Kazhdan-Lusztig
positivity conjectures. (Property (1) follows because the coefficient of $H_y$ in
$\ch([B])$ is given by the graded dimension of a certain hom
space. Property (2) follows because $\mu_{x,y}^z$ gives the
graded multiplicity of $B_z$ as a summand in $B_x \otimes_R B_y$.)

%



The geometric techniques used to understand the
Kazhdan-Lusztig basis yield another remarkable property of the
structure constants $\mu_{x,y}^z$. Using duality, one can show that $\mu_{x,y}^z$ is preserved under swapping $v$ and $v^{-1}$. The quantum numbers
\[
[m] := \frac{v^{m}-v^{-m}}{v - v^{-1}} = v^{-m+1} + v^{-m + 3} + \dots
  + v^{m-3} + v^{m-1} \in \ZM[v^{\pm 1}]
\]
for $m \ge 1$ give a $\ZM$-basis for those elements of $\ZM[v^{\pm
  1}]$ preserved under swapping $v$ and $v^{-1}$. A folklore
conjecture states:\footnote{Unimodality is stated as a 
  question in
  \cite[\S~5.1]{Fokko}, however experts assure the authors that the
  conjecture is much older. In \cite{Fokko}
  positivity properties (2) and (3) are checked for $W$ a finite
  reflection group of $H_4$ by computer (almost three
  trillion polynomials $\mu_{x,y}^z$ need to be computed!). For $H_4$,
  property
  (1) had already been checked by Alvis \cite{Alvis} in 1987. In
  \cite[\S~5.2]{Fokko}
  it is incorrectly stated that the unimodality conjecture is open for
Weyl groups.}
\begin{enumerate}
\item[3)] (``unimodality of structure constants'') if we write
  $\mu_{x,y}^z = \sum_{m \ge 1} a_m [m]$, then $a_m \ge 0$ for all $m$.
\end{enumerate}
(In other words, each $\mu_{x,y}^z$ is the
character of a finite-dimensional $\mathfrak{sl}_2(\CM)$-representation.)

In geometric settings unimodality follows from the
\emph{relative hard Lefschetz theorem} of \cite{BBD}. Recall that the
relative hard Lefschetz theorem states that if $f : X \to Y$ is a
projective morphism of complex algebraic varieties and if $\eta$ is a
relatively ample line bundle on $X$ then for all $i \ge 0$, $\eta$
induces an isomorphism:
\[
\eta^i : {}^p \HC^{-i}( Rf_* IC_X) \simto {}^p \HC^{i}( Rf_* IC_X).
\]
(Here $IC_X$ denotes the intersection cohomology complex on $X$ and
${}^p\HC^i$ denotes perverse cohomology.) In this paper we prove
unimodality for all Coxeter groups, by adapting the relative hard
Lefschetz theorem to the context of Soergel bimodules.

Inside the category of Soergel bimodules we consider the full
subcategory $\pBC$ consisting of direct sums of the indecomposable
self-dual bimodules $B_x$
without shifts. We call $\pBC$ the subcategory of \emph{perverse
  Soergel bimodules}. Soergel's conjecture implies that each $B \in
{}^p \BC$ admits a canonical isotypic decomposition
\[
B = \bigoplus_{x \in W} V_x \otimes_\RM B_x
\]
for certain real (degree zero) vector spaces $V_x$. If a Soergel
bimodule is not perverse, its decomposition into indecomposable
summands of the form $B_x(i)$ is not canonical. However, there is a
canonical filtration on any Soergel bimodule called the \emph{perverse
  filtration}, whose $i$-th subquotient has
indecomposable summands of the form $B_x(-i)$ for some $x \in
W$. Taking the subquotients of this filtration and shifting them
appropriately, one obtains for each $i$ the perverse cohomology
functor 
\[
H^i : \BC \to {}^p \BC.
\]
Any degree $d$ map $B \to B'(d)$ induces a map $H^i(B) \to H^{i+d}(B')$ on
perverse cohomology.

\begin{remark}
  The category $\BC$ is an analogue of semi-simple complexes, $\pBC$
  is an analogue of the category of semi-simple perverse sheaves and
  $H^i$ is an analogue of the perverse cohomology functor.
\end{remark}


This main result of this paper is the following:

\begin{thm}(Relative hard Lefschetz for Soergel bimodules) \label{thm:rhl}
 Let $x, y \in W$ be arbitrary and fix $\rho \in \hg^*$ dominant
 regular (i.e. $\langle \rho, \alpha_s^\vee \rangle > 0$ for all $s
 \in S$). The map
 \begin{align*}
\eta : B_x\otimes_RB_y &\to B_x\otimes_RB_y[2] \\
b \otimes b'  &\mapsto b \otimes \rho b'   = b \rho \otimes b'
 \end{align*}
induces an isomorphism (for all $i \ge 0$)
\[
\b^i : \pH^{-i}(B_x \otimes_R B_y) \simto \pH^i(B_x\otimes_R B_y).
\]
\end{thm}

\begin{remark}
  A stronger version of the above theorem, involving iterated tensor products
  of indecomposable Soergel bimodules of arbitrary length is still open (see Conjecture
  \ref{conj:HRgeneral}). It is amusing that establishing Conjecture
  \ref{conj:HRgeneral} for Bott-Samelson bimodules (i.e., when all
  $x_i \in S$, in the notation of Conjecture \ref{conj:HRgeneral})
  was the authors' original plan of attack to settle Soergel's
  conjecture. This
  remains a very interesting Hodge theoretic statement that we
  \emph{cannot} prove!
\end{remark}

As was true in our
previous work on hard Lefschetz type theorems for Soergel bimodules \cite{EWHodge, WL}, the inductive proof we use to
establish our main theorem actually requires proving a stronger statement, analogous to the relative Hodge-Riemann bilinear relations \cite{dCM2}. That is, we must calculate the
signatures of certain forms on the multiplicity spaces of $H^{-i}(B_x
\otimes_R B_y)$, see Theorem \ref{thm:mainthm}. The following is an immediate consequence of Theorem \ref{thm:rhl}:

\begin{cor}
  The structure constants $\mu_{x,y}^z$ of multiplication in the
  Kazhdan-Lusztig basis are unimodal.
\end{cor}

Relative hard Lefschetz for Soergel bimodules also has important
consequences for certain tensor categories associated to cells in
Coxeter groups. Recall that to any two sided cell $\cell \subset W$ in a finite or affine
Weyl group Lusztig has associated a tensor category, which
categorifies the $J$-ring of $\cell$.  These
categories (for finite Weyl groups) are fundamental for the representation theory of finite
reductive groups of Lie type: by results of Bezrukavnikov, Finkelberg
and Ostrik \cite{BFO2} and Lusztig \cite{LusztigTC}, their (Drinfeld) centers are
equivalent to the braided monoidal category of unipotent character
sheaves corresponding to $\cell$.

Given any two sided cell $\cell \subset W$ in an arbitrary Coxeter
group Lusztig has generalised his construction to yield a monoidal category
$\JC$. (Note that $\JC$ is only ``locally unital'' unless $\cell$ contains
finitely many left cells, and the existence of a unit relies on a
conjecture in general, see Remark \ref{rem:unit}.) In the last section of this paper we explain why Theorem
\ref{thm:rhl} implies that $\JC$ is rigid and pivotal (see Theorem
\ref{thm:rigid}). (The rigidity was conjectured by Lusztig
\cite[\S~10]{LusztigTC} when $W$ is finite). This is an important step towards the study
of ``unipotent character sheaves'' associated to any Coxeter
system.

By a theorem of \cite{Mu, ENO}, rigidity of $\JC$ implies that the
(Drinfeld) center of $\JC$ is a modular tensor category. We expect
cells in non-crystallographic Coxeter groups to provide many
new examples of modular tensor categories (see \cite[5.4]{Ostrik}).

\subsection{Acknowledgements} That relative hard Lefschetz is
tied to the rigity of Lusztig's categorifications of the $J$-ring was
suggested to us by Victor Ostrik. We would like to thank him, as
well as Roman Bezrukavnikov, George Lusztig and Noah Snyder for useful
discussions.

%
\section{Background}
\label{sec-background}
%

\subsection{Soergel bimodules and duality}

Let $\hg$ be an $\RM$-linear realization of the Coxeter system
$(W,S)$, as in \cite[\S~2]{Soe3}. Thus $\hg$ is a finite-dimensional $\RM$-vector space, equipped with
linearly independent subsets of
\emph{roots} $\{ \a_s\}_{s \in S} \subset \hg^*$ and \emph{coroots}
$\{ \a_s^\vee\}_{s \in S} \subset \hg$, such that
\[
\langle \alpha_s, \alpha_t^\vee \rangle = -2 \cos (\pi/m_{st}),
\]
where $m_{st}$ denotes the order (possibly $\infty$) of $st \in
W$.\footnote{The choice of roots and coroots plays a
    significant role in this paper, but only up to positive rescaling;
    what is important (in order that we may cite certain results from
    \cite{Soe3} and \cite{EWHodge}) is that our representations is reflection faithful
    \cite{Soe3} and that there be a well-defined notion of positive
    roots. If the reader prefers, they may also take the representation given
    by a realisation of a generalised Cartan matrix.} We have an action of $W$ on $\hg$ given by the formula \[
s(v) = v - \langle \a_s, v \rangle
\a_s^\vee, \] for all $s \in S$ and $v \in \hg$. The contragredient action of $W$ on $\hg^*$ is defined by an analogous formula \[ s(f) = f - \langle f, \a_s^\vee \rangle \a_s, \] for all $s \in S$ and $f \in \hg^*$.

Let $R$ be the ring of polynomial functions on $\hg$, graded so that the linear terms $\hg^*$ have degree 2. It comes equipped with an action of $W$. Define a graded $R$-bimodule
\[ B_s = R \ot_{R^s} R (1) \] 
for each $s \in S$, where $R^s$ denotes the $s$-invariant polynomial
subring. We use the standard convention for grading shifts, so that
the $(1)$ above indicates that the minimal degree element $1 \ot 1$
lives in degree $-1$. Given two graded $R$-bimodules $B, B'$ their tensor product over $R$ is denoted $BB' := B
\otimes_R B'$. For a sequence $\un{w} = (s_1, s_2, \ldots, s_d)$ with $s_i \in S$, the tensor product
\[ BS(\un{w}) = B_{s_1} B_{s_2} \ldots B_{s_d} \]
is called a \emph{Bott-Samelson bimodule}.

Soergel proved in \cite{Soe3} that, when $\un{x}$ is a reduced
expression for an element $x \in W$, there is a unique indecomposable
direct summand $B_x \sumset BS(\un{x})$ which is not isomorphic to a
summand of a shift of any Bott-Samelson bimodule corresponding to a
shorter reduced expression. Moreover, this summand does not depend on
the reduced expression of $x$, up to non-canonical
isomorphism. (Using the main theorem of \cite{EWHodge}
  one can make this isomorphism canonical.) Note that
the two notations for $B_s$ agree.

Let $\BC$ denote the full subcategory of graded $R$-bimodules whose objects are finite direct sums of grading shifts of summands of Bott-Samelson bimodules. The objects in this category
$\BC$ are known as \emph{Soergel bimodules}, and the bimodules $\{B_x\}_{x \in W}$ give a complete list of non-isomorphic indecomposable objects up to grading shift. Because Bott-Samelson
bimodules are closed under tensor product, $\BC$ is as well, and inherits its monoidal structure from $R$-bimodules.

If $B$ is a Soergel bimodule we will often use the symbol $B$ to denote
the identity morphism on $B$. For example, if $f : B' \to B''$ is a morphism then $Bf : BB' \to BB''$ denotes the tensor product of the identity on $B$ with $f$. Similarly, given $r \in R$
of degree $m$, $rB$ (resp. $Br$) denotes the morphism $B \to B(m)$ given by left (resp. right) multiplication by $r$.

For two Soergel bimodules $B$ and $B'$, we let $\Hom(B,B')$ denote the degree zero homomorphisms of $R$-bimodules, and write 
\[
\Hom^\bullet(B, B') = \bigoplus_{m \in \ZM} \Hom(B, B'(m))
\]
for the graded vector space of bimodule homomorphisms of all degrees. A morphism $f \in \Hom(B,B'(m))$ is said to be a \emph{degree $m$} morphism from $B$ to $B'$. By a theorem of
Soergel \cite[Theorem 5.15]{Soe3}, $\Hom^\bullet(B, B')$ is free of finite rank as a left or right $R$-module.

Given a Soergel bimodule $B \in \BC$ its dual is
\[
\DM B := \Hom^{\bullet}_{-R}(B, R)
\]
where $\Hom^{\bullet}_{-R}$ denotes the graded vector space of right $R$-module homomorphisms of all
degrees. We make $\DM B$ into an $R$-bimodule via
$r_1fr_2(b) = f(r_1br_2)$. Because $\DM BS(\un{w}) \cong BS(\un{w})$,
the functor $\DM$ descends to a contravariant equivalence of $\BC$. By the defining property of the indecomposable bimodule $B_x$, we must also have $\DM B_x \cong B_x$. As usual, $\DM \DM B \cong B$ canonically, for any Soergel bimodule $B$.

A \emph{pairing} on two Soergel bimodules $B, B'$ is a homogeneous bilinear form
\[
\langle -, - \rangle : B \times B' \to R
\]
such that $\langle rb, b' \rangle = \langle b, rb' \rangle$ and
$\langle br, b' \rangle = \langle b, b'r\rangle = \langle b, b'
\rangle r$ for all $b \in B, b' \in B$ and $r \in R$. (Note the
asymmetry in the conditions on the left and right
$R$-actions.\footnote{This is the convention used in \cite{EWHodge}. The opposite
  convention is used in \cite{WL}.}) The homogeneous condition states that $\deg b + \deg b' = \deg \langle b, b' \rangle$.
A pairing induces bimodule morphisms $B \to \DM B'$ and $B' \to
\DM B$. We say that a pairing is non-degenerate if one (or
equivalently both) of these morphisms is an
isomorphism.\footnote{Warning: this is stronger than the condition
  $\langle b, B \rangle = 0 \Rightarrow
  b = 0$.}

A \emph{(non-degenerate) form} on a Soergel bimodule is a
(non-degenerate) pairing
\[
\langle -, - \rangle : B \times B \to R
\]
which is in addition symmetric: $\langle b, b' \rangle = \langle b', b
\rangle$ for all $b, b' \in B$. A \emph{polarized Soergel bimodule} is a pair $(B,
\langle -, - \rangle_B)$ where $B \in \BC$ is a Soergel bimodule and
$\langle -, -\rangle_B$ is a non-degenerate form, in which case 
$\langle -, - \rangle_B$ is the \emph{polarization}.

Given a map $f : B \to B'(m)$ between polarized Soergel bimodules its
\emph{adjoint} is the unique map $f^* : B' \to B(m)$ such that
\[
\langle f(b), b' \rangle_B = \langle b, f^*(b') \rangle_{B'} \quad
\text{for all $b \in B, b' \in B'$.}
\]
Equivalently $f^* = \DM f \colon \DM (B'(m)) \to \DM B$, where we use the
polarizations to identify $B = \DM B$, $B'(-m) = \DM (B'(m))$.

\subsection{Perverse cohomology and graded multiplicity spaces} \label{sec:notes}
All morphisms between indecomposable self-dual Soergel bimodules are
of non-negative degree, and those of degree zero are isomorphisms. That is:
\begin{gather}
  \label{eq:deltaxy}
  \Hom(B_x, B_y)  = \begin{cases} \RM & \text{if $x = y$,} \\ 0 &
  \text{otherwise.} \end{cases}   \\
  \label{eq:perv_vanish}
  \Hom(B_x, B_y(m)) = 0 \quad \text{for $x, y \in W$ and $m < 0$,}
\end{gather}
These fundamental Hom-vanishing statements are equivalent to Soergel's
conjecture (see the paragraph following \cite[Theorem 3.6]{EWHodge}).

A Soergel bimodule $B$ is \emph{perverse} if it is isomorphic to a
direct sum of indecomposable bimodules $B_x$ without shifts. We denote
by $\pBC$ the full subcategory of perverse Soergel bimodules. As a
consequence of \eqref{eq:perv_vanish}, any perverse Soergel bimodule
admits a canonical decomposition
\begin{gather} \label{eq:pervdecomp}
B = \bigoplus_{x \in W} V_x \otimes_{\RM} B_x
\end{gather}
for some finite dimensional real vector spaces $V_x$. (Concretely, one
has $V_x = \Hom(B_x, B)$.) The rest of this section is dedicated to understanding what replaces this multiplicity space $V_x$ in case the bimodule $B$ in question is not perverse.

By the classification of indecomposable bimodules, every Soergel
bimodule splits into a direct sum of shifts of perverse bimodules, but
this splitting is not canonical. However, it is a 
consequence of \eqref{eq:deltaxy} and \eqref{eq:perv_vanish} that $B$
admits a unique functorial (non-canonically split) filtration, the
\emph{perverse filtration}, whose subquotients isomorphic to a shift
of a perverse Soergel bimodule, see \cite[\S 6.2]{EWHodge}. Before 
discussing the details, it is worth illustrating this subtle point in
examples. 

\begin{ex} The Bott-Samelson bimodule $B_s B_s$ is isomorphic to $B_s(+1) \oplus B_s(-1)$. The degree $-1$ projection map, that is, the map $B_s B_s \to B_s(-1)$, is canonical up to a
scalar. After all, it is easy to confirm from \eqref{eq:deltaxy} and \eqref{eq:perv_vanish} that $\Hom^{\bullet}(B_s B_s, B_s) \cong \Hom^{\bullet}(B_s(+1) \oplus B_s(-1), B_s)$ is zero
in degrees $\le -2$, and is one-dimensional in degree $-1$. The same can be said about the degree $-1$ inclusion map, that is, the map $B_s(+1) \to B_s B_s$. However, the degree $+1$
projection map $B_s B_s \to B_s(+1)$ (resp. the degree $+1$ inclusion map $B_s(-1) \to B_s B_s$) is not canonical; adding to it an $R$-multiple of the degree $-1$ projection map will
give another valid projection map. Said another way, $B_s(+1)$ is a canonical submodule, and $B_s(-1)$ a canonical quotient, and this filtration of $B_s B_s$ splits, but not
canonically. \end{ex}

\begin{ex} \label{ex:stst} Suppose that $W$ is of type $A_2$ with
  simple reflections $\{ s, t \}$. Then $BS(stst)$ is isomorphic to $B_{sts}(-1) \oplus B_{sts}(+1) \oplus B_{st}$. The degree $-1$
projection map to $B_{sts}(-1)$ is canonical. However, the degree $0$ projection map to $B_{st}$ is not canonical! The morphism space $\Hom(BS(stst), B_{st})$ is two-dimensional; it
has a one-dimensional subspace arising as the composition of the canonical projection to $B_{sts}(-1)$ followed by a non-split map $B_{sts}(-1) \to B_{st}$, and any morphism not in this
one-dimensional subspace will serve as a projection map to $B_{st}$. This example is meant to loudly proclaim that even what appears to be an ``isotypic component,'' such as the summand
$B_{st}$ which is the only one of its kind, is not canonically a direct summand, owing to the presence of other summands with lower degree shifts. \end{ex}

For any $i \in \ZM$, define $\BC^{\le i}$ (resp. $\BC^{>i}$) to be the full additive
subcategory of $\BC$ consisting of bimodules which are isomorphic to
direct sums of $B_x(m)$ with $m \ge -i$ (resp. $m < -i$). In formulas:
\begin{gather*}
  \BC^{\le i} := \langle B_x(m) \; | \; x \in W, m \ge -i
  \rangle_{\oplus, \cong}, \\
  \BC^{> i} := \langle B_x(m) \; | \; x \in W, m < -i
  \rangle_{\oplus, \cong}.
\end{gather*}
Similarly we define $\BC^{<i}$ and $\BC^{\ge i}$. We have $\pBC =
B^{\ge 0} \cap B^{\le 0}$.
We can rephrase \eqref{eq:perv_vanish} as the statement:
\begin{gather}
   \label{eq:perv_vanish2}
  \Hom(\BC^{\le i}, \BC^{>i}) = 0.
\end{gather}

Any Soergel bimodule $B$ admits a unique \emph{perverse filtration}
\[
\dots \subset \tau_{\le i}B \subset \tau_{\le i+1}B \subset \dots
\]
by split inclusions such that $\tau_{\le i} B \subset \BC^{\le i}$ and $B / \tau_{\le i}B \in \BC^{>i}$, see \cite[\S 6.2]{EWHodge}. This is a direct consequence of \eqref{eq:perv_vanish2}. If $f \colon B \to B'$ is a morphism then $f(\tau_{\le i} B) \subset \tau_{\le i} B'$. We have:
\begin{equation}
  \label{eq:pshift}
  \tau_{\le i}(B(m)) = (\tau_{\le i + m}B)(m).
\end{equation}

Dually, every Soergel bimodule has a unique \emph{perverse
  cofiltration}
\[
\dots \onto \tau_{\ge i}B \onto \tau_{\ge i + 1}B \onto \dots
\]
where every arrow is a split surjection, each $\tau_{\ge i}B \in
\BC^{\ge i}$ and the kernel of $B \onto \tau_{\ge i}B$ belongs to
$B^{< i}$. We have:
\begin{equation}
  \label{eq:dualcofilt}
\DM(\tau_{\ge i} B) = \tau_{\le -i} (\DM B).
\end{equation}

The \emph{perverse cohomology} of a Soergel bimodule $B$ is
\[
H^i(B) := (\tau_{\le i} B / \tau_{< i}B) (i).
\]
(The shift $(i)$ is included so that $H^i(B)$ is perverse.) Applying
\eqref{eq:pervdecomp} we
obtain canonical isotypic decompositions
\[
\pH^i(B) = \bigoplus_{z \in W} H^i_z(B) \otimes_\RM B_z.
\]
for certain finite dimensional vector spaces $H_z^i(B)$. We have a non-canonical isomorphism
\[
B \cong \gr B := \bigoplus_{i \in \ZM} H^i(B)(-i)
\]
and canonical isomorphisms
\[
\gr B = \bigoplus_{i, z}  H_z^i(B) \otimes B_z(-i) = \bigoplus_{z}  H^\bullet_z(B)
\otimes_\RM B_z
\]
where $H^\bullet_z(B)$ denotes the graded vector space $\bigoplus
H_z^i(B)$. Below we call the graded vector spaces $H^\bullet_z(B)$
\emph{multiplicity spaces}.

\begin{remark} To reiterate the point made in Example \ref{ex:stst}:
  in general, it is not possible to produce separate multiplicity spaces $H^\bullet_z(B)$, for different $z \in W$, without first passing to the associated graded of the perverse filtration. \end{remark}

Let $B, B'$ be Soergel bimodules and $f : B \to B'(m)$ a
morphism. Then by \eqref{eq:perv_vanish2} and \eqref{eq:pshift} we have
\[
f(\tau_{\le i} B) \subset \tau_{\le i}(B'(m)) = (\tau_{\le i+
  m}B')(m).
\]
Thus $f$ induces a map
\begin{equation*}
  f : H^i(B) \to H^{i+m}(B')
\end{equation*}
of Soergel bimodules, and hence a degree $m$ map $\gr f$ from $\gr B$ to $\gr B'$. For any $z \in W$ this induces a map
\begin{equation*}
  \gr_z f : H_z^\bullet(B) \to H_z^{\bullet+ m}(B')
\end{equation*}
of graded vector spaces. To simplify notation, we use $f$ to denote
all these maps: $f$, $\gr f$, $\gr_z f$ for all $z \in W$. We refer to
the maps $\gr f$ and $\gr_z f$ as the \emph{maps induced on perverse
  cohomology}.

The following triviality is important later:

\begin{lem} \label{lem:pzero}
  If $f : B \to B'(m)$ is a map such that, for all $i \in \ZM$,
\[
f(\tau_{\le i}B) \subset
  \tau_{\le i-1}(B'(m))\]
then $f$ induces the zero map on perverse
  cohomology. In particular, this applies to the map
    given by left or right
  multiplication by any positive-degree polynomial in $R$ on a Soergel
  bimodule $B$.
\end{lem}

\begin{proof}
  Only the second sentence requires proof. The perverse filtration is a
  filtration by $R$-bimodules. If $r \in R$ is homogenous of
  degree $d > 0$ then multiplication by $r$ on the left (resp. right)
  induces a map (see \eqref{eq:pshift}) 
\[
\tau_{\le i}B \to (\tau_{\le i}B)(d) =(\tau_{ \le i-d}(B))(d).
\]
Therefore, the hypothesis of the lemma applies to multiplication by $r$.
\end{proof}

\subsection{Polarizations of Soergel bimodules}
In \cite[\S 3.4, see also Corollary 3.9]{EWHodge}, the Bott-Samelson
bimodule $BS(\un{w})$ was equipped with a non-degenerate form called
the \emph{intersection form}. By restriction, one obtains a form on
any summand of a Bott-Samelson bimodule. By \cite[Lemma 3.7]{EWHodge},
there is, up to an invertible scalar, a unique non-zero form on an indecomposable Soergel bimodule
$B_x$ (this statement is equivalent to Soergel's
  conjecture), and it is non-degenerate. Thus,
letting $\un{x}$ be any reduced expression for $x$, the restriction of
the intersection form to $B_x \sumset BS(\un{x})$ is non-zero, hence
is non-degenerate and hence is a polarization of $B_x$.
 For all $x \in W$ we fix a reduced expression $\un{x}$ of $x$ and an
 embedding $B_x \subset BS(\un{x})$, and hence
a polarization $\langle -, -\rangle_{B_x}$ on $B_x$. We refer to
$\langle -, -\rangle$ as the \emph{intersection form} on $B_x$.
The intersection form has the following important positivity property:

\begin{lem}[\cite{EWHodge}, Lemma 3.10] If 
 $\rho \in \hg^*$ dominant
 regular (i.e. $\langle \rho, \alpha_s^\vee \rangle > 0$ for all $s
 \in S$) and $b \in B_x$ is any non-zero element of degree $-\ell(x)$
 then
\[
\langle b, \rho^{\ell(x)} b \rangle > 0. \footnote{The
    reader should not forget that $\langle b, b' \rangle$ is, in
    general, an element of the ring $R$. Here, for degree reasons, one
    obtains a degree zero element of $R$, hence an element of $\RM$.}
\]
\end{lem}

\begin{remark} \label{rmk:posscalar} This lemma and the
    discussion of the previous paragraph implies that the intersection
    form on $B_x$ does not depend on the choice of reduced expression
    $\un{x}$ or the choice of embedding $B_x \subset BS(\un{x})$, up
    to multiplication by a positive scalar. \end{remark}





Given any polarized Soergel bimodule $B$, it is explained in \cite[\S 3.6]{EWHodge} how to produce a polarization on $BB_s$, called the \emph{induced form}. Moreover, if $B = B_x$ is
given its intersection form (i.e. the form restricted from our fixed inclusion $B_x \sumset BS(\un{x})$ for a reduced expression) then the induced form on $B B_s$ agrees with the form
restricted from the inclusion $B_x B_s \sumset BS(\un{x} s)$. This is because the intersection form on any Bott-Samelson bimodule $BS(\un{w})$ is constructed by being repeatedly induced from the canonical form on $BS(\emptyset) = R$. Let us generalize this notion of induced forms.

If $B$ and $B'$ are two polarized Soergel bimodules, we define a form
on $B B'$ by the formula
\begin{equation} \label{eq:inducedform} \langle b \ot b', c \ot c'
  \rangle_{B B'} := \langle (\langle b,c \rangle_B) \cdot b', c'
  \rangle_{B'}= \langle b', (\langle b,c \rangle_B) \cdot c' \rangle_{B'}. \end{equation}
It is an exercise to confirm that the induced form on $B B_s$ is
defined precisely in this fashion.

\begin{lem}[\cite{WL}, \S 6.4]  The induced form on $B B'$ is
  non-degenerate, and thus is a polarization of $B
  B'$.
\end{lem}

By iteration, we have an induced form on any tensor product of the form $B_{x_1} B_{x_2} \cdots B_{x_m}$, which we continue to call the \emph{intersection form}. One could also view
$B_{x_1} \cdots B_{x_m}$ as a summand (via the tensor products of our
fixed embeddings) of a Bott-Samelson bimodule $BS(\un{w})$, where
$\un{w}$ is a concatenation of our chosen reduced expression for each
$x_i$. The induced form agrees with the restriction of the
intersection form on $BS(\un{w})$ to this summand.
All tensor products of the form $B_{x_1} B_{x_2} \dots B_{x_m}$ are
always assumed to be polarized with respect to their intersection
form.

Let $(B, \langle -, -\rangle_B)$ be a polarized Soergel bimodule. If
$B$ is also perverse then by considering the the isotypic
decomposition (see \eqref{eq:pervdecomp})
\[ B =
\bigoplus_{x \in W} V_x \otimes_{\RM} B_x
\]
and the associated map $B
\to \DM B$, we see that $\langle -, -\rangle_B$ is orthogonal for this
decomposition. Moreover, $\langle -, -\rangle_B$ is determined by
symmetric forms $\langle - , -\rangle_{V_x}$ on each vector space
$V_x$ (i.e. if $v, v' \in V_x$ and $b, b' \in B_x$ then $\langle v
\otimes b, v' \otimes b' \rangle = \langle v, v' \rangle_{V_x} \langle
b, b' \rangle_{B_x}$). We say that $B$ is \emph{positively
  polarized} if $B = 0$ or the following conditions are satified:
\begin{enumerate}
\item $B$ is perverse and vanishes in even or odd degree (because
  $B_x$ is non-zero in degree $-\ell(x)$, the second
  condition is equivalent to the existence of $q \in \{0, 1\}$ such that $V_x =
  0$ for all $x$ with $\ell(x)$ of the same parity as $q$);
\item Let $z \in W$ denote the element of maximal length in $W$ such
  that $V_z \ne 0$.  If $V_y \ne 0$ then $\langle
  -, -\rangle_{V_y}$ is $(-1)^{(\ell(z) - \ell(y))/2}$ times a
  positive definite form, for all $y \in W$.
\end{enumerate}

The canonical example of a positively polarized Soergel bimodule is
given by the following lemma:

\begin{lem}[\cite{EWHodge}, Proposition 6.12] \label{ys}
  Suppose that $y \in W$ and $s \in S$ with $ys > y$ (resp. $sy > y$). Then 
  $B_yB_s$ (resp. $B_sB_y$), equipped with its interesection form, is
  positively polarized.
\end{lem}

\subsection{Forms on multiplicity spaces} 

Assume that $(B, \langle -, -\rangle)$ is a polarized Soergel
bimodule. If we interpret $\langle -, - \rangle$ instead as an isomorphism
\begin{gather*}
  f : B \simto \DM(B)
\end{gather*}
then we deduce from the functoriality of the perverse filtration that:
\begin{gather}
  f(\tau_{\le i}) \subset \tau_{\le i}(\DM B)
  \stackrel{\eqref{eq:dualcofilt}}{=} \DM( \tau_{\ge -i} B), \label{eq:s1} \\
  f \text{ induces an isomorphism }H^i(B) \simto H^i(\DM B) = \DM H^{-i}(B). \label{eq:s2}
\end{gather}
Statement \eqref{eq:s1} is equivalent to saying that $\langle \tau_{\le i}
B, \tau_{<-i} B \rangle = 0$ and hence that $\langle -, - \rangle$
induces a pairing of Soergel bimodules
\begin{equation}
  \label{eq:ipairing}
\langle -, - \rangle : H^i(B) \times H^{-i}(B) \to R,  
\end{equation}
and \eqref{eq:s2} tells us that this pairing is non-degenerate. By
\eqref{eq:deltaxy} the canonical decompositions
\[
H^i(B) = \bigoplus_{z \in W} H^i_z(B) \otimes_{\RM} B_z \quad
\text{and} \quad
H^{-i}(B) = \bigoplus_{z \in W} H^{-i}_z(B) \otimes_{\RM} B_z
\]
are orthogonal with respect to $\langle -, -\rangle$ (i.e. $\langle \g
\otimes b, \g' \otimes b' \rangle = 0$ for $\g \otimes b \in H^i_z(B)
\otimes_\RM B_z$ and $\g' \otimes b' \in H^{-i}_{z'}(B)
\otimes_\RM B_{z'}$ if $z \ne z'$). Applying \eqref{eq:deltaxy} again
we conclude that \eqref{eq:ipairing} is completely determined by the
non-degenerate bilinear pairing on the vector spaces
\begin{equation} \label{eq:izpairing}
H^i_z(B) \times H^{-i}_z(B) \to \RM
\end{equation}
for all $z \in W$. To be precise, given $v \in H^i_z(B)$ and $v' \in H^{-i}_z(B)$, this pairing \eqref{eq:izpairing} is defined so that, for all $b,b' \in B_z$, one has
\begin{equation} \label{eq:pairingexplicit}
\langle v \ot b, v' \ot b' \rangle = \langle v,v' \rangle \langle b,b' \rangle.
\end{equation}
The left hand side is (a summand of) the pairing in
\eqref{eq:ipairing} between $H^i(B)$ and $H^{-i}(B)$, and the right
hand side is the pairing in \eqref{eq:izpairing} multiplied by the
intersection form on $B_z$.

Reassembling this data, we conclude that $\langle -, - \rangle$ descends to a symmetric
non-degenerate form
\[
\langle -, -\rangle : \gr B \times \gr B \to R.
\]
and that this form is determined by the symmetric
non-degenerate graded bilinear forms
\[
\langle -, - \rangle : H_z^\bullet(B) \times H_z^\bullet(B) \to \RM
\]
on multiplicity spaces for all $z \in W$.

Here is another important triviality:

\begin{lem} \label{lem:canonicalpairingis1} Let $B = BS(\un{x})$ be a
  Bott-Samelson bimodule associated to a reduced expression $\un{x}$
  for an element $x \in W$, polarized with respect to its intersecton
  form.  The summand $B_x$ appears with multiplicity one having
no grading shift, so that $H^\bullet_x(B) = \RM$ in degree
zero. Up to a positive scalar, the form $H^0_x(B) \times H^0_x(B) \to \RM$ is
just the standard form, with $\langle 1, 1\rangle = 1$. \end{lem}

\begin{proof} This follows immediately from
  \eqref{eq:pairingexplicit}, because the intersection
    form on $BS(\un{x})$ restricts to a positive multiple of the
    intersection form on $B_x$ (see Remark
    \ref{rmk:posscalar}). \end{proof}

%
\section{Relative hard Lefschetz and Hodge-Riemann}
\label{sec-relative}
%

\subsection{Statement}

We fix once and for all a dominant regular $\rho \in \hg^*$, that is, an element such that $\langle \rho, \a_s^\vee \rangle \ge 0$ for all $s \in S$.

Let $\bfx := (x_1, \dots, x_m)$ be a sequence of elements in $W$, and fix scalars $\bfa := (a_1, \dots, a_{m-1}) \in
\RM^{m-1}$. Consider the operator
\begin{gather*}
L_\bfa : B_{x_1}B_{x_2} \dots B_{x_m} \to B_{x_1}B_{x_2} \dots B_{x_m}(2) \\
L_\bfa = a_1 B_{x_1} \rho B_{x_2} \dots B_{x_m} + a_2 B_{x_1} B_{x_2} \rho \dots B_{x_m} + \dots + a_{m-1} B_{x_1} B_{x_2} \dots \rho B_{x_m}.
\end{gather*}
In words, $L_\bfa$ is the sum of the operators of multiplication by $a_i \rho$ in the gap between $B_{x_i}$ and $B_{x_{i+1}}$.

We have explained that to any $z \in W$ we may
associate a graded vector space
\begin{gather*}
V^\bullet :=  H_z^\bullet(B_{x_1}B_{x_2} \dots B_{x_m})
\end{gather*}
equipped with
\begin{enumerate}
\item a symmetric graded non-degenerate form $\langle -,
  -\rangle_{V^\bullet}$ obtained from the intersection form on $B_{x_1}
  \dots B_{x_m}$;
\item a degree two Lefschetz operator $L_{\bfa} : V^\bullet
  \to V^{\bullet + 2}$ obtained by taking perverse cohomology
  of $L_\bfa$.
\end{enumerate}

\begin{remark} \label{rem:internal} The operator $L_{\bfa}$ involves
  only internal multiplication by polynomials. One could also consider
  the Lefschetz operator $L_{\bfa} + a_0 \rho \cdot (-) + a_m (-)
  \cdot \rho$ which includes multiplication on the left and
  right. However, as observed in Lemma \ref{lem:pzero}, left and right
  multiplication by polynomials act trivially on perverse cohomology,
  so this does not affect the degree $2$ operator on $V^\bullet$. \end{remark}

We say that $L_\bfa$ \emph{satisfies relative hard Lefschetz}
  if for any $d \ge 0$, $L_\bfa$ induces an isomorphism:
\[
L_\bfa^d : V^{-d} \simto V^d.
\]
We say that $L_\bfa$ \emph{satisfies relative Hodge-Riemann} if
$L_\bfa$ satisfies relative hard Lefschetz and the restriction of the
Lefschetz form $(v, v') :=  \langle v, L_\bfa^dv' \rangle_{V^\bullet}$ on $V^{-d}$ to
\[
P^{-d} := \ker L_\bfa^{d+1} : V^{-d} \to V^{d+2}
\]
is $(-1)^{\e(\bfx, z, d)}$-definite, for all $d \ge 0$, where
\[
\e(\bfx, z, d) := \frac{1}{2} \left ( \sum_{i = 1}^m\ell(x_i) - \ell(z)
  - d \right ).
\]
Note that relative hard Lefschetz and relative Hodge-Riemann are both statements about $H_z^\bullet$ which are required to hold for all $z \in W$.

\begin{remark}
 The sign $(-1)^{\e(\bfx, z, d)}$ might appear
  mysterious. The following is a useful mnemonic. Set $B := B_{x_1}
  \dots B_{x_m}$ and consider the
  finite dimensional graded vector space
\[
\overline{B} := B\otimes_R \RM.
\]
We have a non-canonical isomomorphism
\[
\overline{B}  \cong \bigoplus_{z \in W} H_z^{\bullet}(B) \otimes \overline{B}_z.
\]
Now $\e(\bfx, z, d)$ has the following meaning: it is half the difference
between the smallest non-zero degree in $H^{-d}_z(B) \otimes_{\RM} \overline{B}_z^{-\ell(z)}$ on
the right hand side (i.e. $-\ell(z) - d$) and the smallest non-zero
degree in $\overline{B}$ (i.e. $-\sum \ell(x_i)$). In this way one may see that the above definition
is compatible with the signs predicted by Hodge theory in the geometric setting (see \cite{dCM2} and
\cite[Theorem 3.12]{WB},
where the signs are made explicit).
\end{remark}

For $x_1, \dots, x_m \in W$ as above we introduce the following
abbreviations:
\begin{gather*}
  RHL(x_1, \dots, x_m) : \begin{array}{c}
L_{\bfa} \text{ satisfies relative hard Lefschetz}  \\
\text{for all $\bfa := (a_1,
                          \dots, a_{m-1}) \in \RM_{>0}^{m-1}$}.
\end{array} \\
  RHR(x_1, \dots, x_m) : \begin{array}{c}
L_{\bfa} \text{ satisfies relative Hodge-Riemann}  \\
\text{for all $\bfa := (a_1,
                          \dots, a_{m-1}) \in \RM_{>0}^{m-1}$}.
\end{array}
\end{gather*}
As always, it is implicitly assumed in these statements that all tensor products of the
  form $B_{x_1} \dots B_{x_m}$ are 
  equipped with their intersection form.

The main theorem of this paper is:

\begin{thm} \label{thm:mainthm}
  For any $x, y \in W$, $RHR(x, y)$ holds.
\end{thm}

\subsection{A conjecture}

\begin{conj} \label{conj:HRgeneral}
  For any $x_1, \dots, x_m \in W$, $RHR(x_1, \dots, x_m)$ holds.
\end{conj}

More generally, relative Hodge-Riemann should hold for any operator of the form
\[ B_{x_1} \rho_1 B_{x_2} \dots B_{x_m} + B_{x_2} B_{x_2} \rho_2 \dots B_{x_m} + \dots + B_{x_1} B_{x_2} \dots \rho_{m-1} B_{x_m}, \]
where $\rho_1, \dots, \rho_{m-1}$ is any sequence of dominant regular
elements. (Such elements span the cone of relatively ample classes in the Weyl group case.) For the conjecture above, one sets $\rho_i = a_i \rho$.

\subsection{Base cases}

\begin{lem} \label{lem:x} $RHL(x)$ and $RHR(x)$ hold, for any $x \in W$. \end{lem}

\begin{proof} The only nonvanishing $H^\bullet_z(B_x)$ occurs when $z = x$, and this multiplicity space is concentrated in degree zero. Thus $RHL(x)$ is trivial, and $RHR(x)$ is
equivalent to the statement that the form $H^0_x(B_x) \times H^0_x(B_x) \to \RM$ is positive definite, which holds by Lemma \ref{lem:canonicalpairingis1}. \end{proof}

\begin{lem} \label{lem:remove1} If $RHL(x_1, x_2, \ldots, x_m)$ holds, then so does $RHL(x_1, \ldots, x_m,\id)$ and $RHL(\id, x_1, \ldots, x_m)$. The same statement can be made for $RHR$.
\end{lem}

\begin{proof} Let us compare $RHL(x_1, x_2, \ldots, x_m)$ and $RHL(x_1, \ldots, x_m,\id)$. Because $B = B_{x_1} \cdots B_{x_m} = B_{x_1} \cdots B_{x_m} B_\id$, the multiplicity spaces
$H^\bullet_z(B)$ being studied are the same. The operator $L_{\bfa}$ on $B$ is different, because in the latter case, one is also permitted to multiply by $a_m \rho$ in the slot before
the final $B_1$. However, this is equal to right multiplication by $a_m \rho$, which acts trivially on perverse cohomology. See Lemma \ref{lem:pzero} and Remark \ref{rem:internal}. Thus
the Lefschetz operators on $H^\bullet_z(B)$ are the same. \end{proof}

To warm up, we consider the first interesting case: $RHR(x,s)$, for $s \in
S$. This
splits into two subcases: $xs<x$ and $xs>x$. Suppose that
$xs>x$. Then $B_x B_s$ is perverse, and so each $H^\bullet_z(B_x B_s)$ is
concentrated in degree $0$ and $RHL(x,s)$ 
holds automatically. In this case $RHR(x,s)$ is equivalent to Lemma \ref{ys}.

Suppose now that $xs < x$. Then $B_x B_s \cong B_x(+1) \oplus B_x(-1)$. The action of $B_x \rho B_s$ on the multiplicity spaces $H_x^\bullet(B_x B_s)$ is independent of $x$ (see Lemma
\ref{lem:operator} below), and can be computed when $x = s$, where it is a simple exercise. (We have been brief here because this computation, expanded upon and in further generality,
comprises the bulk of \S \ref{subsec-signs}.)

\subsection{Structure of the proof}

Let us outline the major steps in the proof of Theorem
\ref{thm:mainthm}, which will be carried out in the rest of this
paper. The proof is by induction on $\ell(x) + \ell(y)$ and then on
$\ell(y)$. More precisely, for integers $M$ and $N$, consider the statements:
\begin{align*}
  X_{M,N} &: \begin{array}{c} \text{$RHR(x',y')$ holds whenever either } \\ \text{(1) $\ell(x') + \ell(y') < M$,} \\
\text{(2) $\ell(x') + \ell(y') = M$ and $\ell(y') \le N$}. \end{array} \\
  Y_{M,N} &: \begin{array}{c} \text{$RHR(x', s, y')$ holds, for all $s \in S$, whenever either} \\ \text{(1) $\ell(x') + \ell(y') + 1 < M$,} \\
\text{(2) $\ell(x') + \ell(y') + 1 =  M$ and $\ell(y') \le N$.}
\end{array}
\end{align*}
(So $M$ always bounds the length of the sequence, and
  $N$ bounds the length of the final factor.)

Certain implications are obvious. For example, $X_{M,M}$ implies
$X_{M,N}$ for all $N \ge M$. Similarly, $X_{M,0}$ is equivalent
to $X_{M-1,N}$ for all $N$, because the only element with $\ell(y') = 0
$ is $y' = \id$ (see Lemmas \ref{lem:x} and
\ref{lem:remove1}). Similar statements hold for $Y_{M,N}$.

So let us fix $M > N \ge 1$. Our goal is to show that $X_{M,N}$ and
$Y_{M,N-1}$ together imply $X_{M,N+1}$ and $Y_{M,N}$.

Let $x, y \in W$ be such that $\ell(x) + \ell(y) = M$ and
$\ell(y) = N+1$. By a weak Lefschetz style argument (Proposition \ref{prop:weaklefschetz1}):
\begin{align} \label{eq:indwl}
  RHR(<x, y) + RHR(x,<y) & \Rightarrow RHL(x,y).
\end{align}


Let us fix $s \in S$ with $sy < y$ and set $\dot{y} := sy$. Again weak Lefschetz style arguments yield (Proposition \ref{prop:weaklefschetz2}):
\begin{align}\label{eq:indwl2}
  RHR(<x,s,\dot{y}) + RHR(x,s,<\dot{y}) & \Rightarrow RHL(x,s,\dot{y})
\end{align}

We now distinguish two cases. If $xs > x$ then an easy limit argument (Proposition \ref{prop:xsbigger}) gives:
\begin{align}\label{eq:ind3}
  RHR(\le xs, \dot{y}) + RHL(x,s,\dot{y}) & \Rightarrow RHR(x,s,\dot{y}).
\end{align}

If $xs < x$ then a more complicated limit argument (Proposition \ref{prop:xssmaller}) allows us to reach essentially the same conclusion:
\begin{align}\label{eq:ind4}
  RHR(x, \dot{y}) + RHL(x,s,\dot{y}) & \Rightarrow RHR(x,s,\dot{y}).
\end{align}

Another limit argument (Proposition \ref{prop:xy}) yields:
\begin{align}\label{eq:ind5}
  RHR(x, s, \dot{y}) + RHL(x,\le y) & \Rightarrow RHR(x,y).
\end{align}
Thus assuming $X_{M,N}$ and $Y_{M,N-1}$ we have concluded that $X_{M,N+1}$ holds.

Finally, if $x, y \in W$ and $t \in S$ is such that $\ell(x) + \ell(y)
+ 1 = M$ and $\ell(y) = N$ then as in \eqref{eq:indwl2} we deduce:
\begin{align}\label{eq:indwlsim}
  RHR(<x,t,y) + RHR(x,t,<y) & \Rightarrow RHL(x,t, y).
\end{align}

If $xt < x$ then we have
\begin{align}\label{eq:ind6}
  RHR(x,y) + RHL(x,t,y) & \Rightarrow RHR(x,t, y).
\end{align}
If $xt > x$ then we have
\begin{align*} \label{eq:ind7}
   RHR(\le xt, y) +  RHL(x,t, y) & \Rightarrow HR(x,t,y)
\end{align*}
Thus assuming $X_{M,N}$ and $Y_{M,N-1}$ we have deduced that $Y_{M,N}$ holds.

Putting these two steps together we deduce:
\[
X_{M,N} + Y_{M,N-1} \Rightarrow X_{M,N+1} + Y_{M,N}.
\]
We conclude by induction that $X_{M,M}, Y_{M,M}$ hold for all $M$. This reduces the proof of the theorem to the propositions listed above.

%
\section{The proof}
\label{sec-proof}
%

\subsection{Hodge-Riemann implies hard Lefschetz} \label{sec:proof1}

In \cite{EWHodge} it was observed that homological algebra in the homotopy category of Soergel bimodules can be used to imitate the weak Lefschetz theorem. This is the key step to
deduce the hard Lefschetz theorem by induction. In this section we show that the same idea is useful for studying relative hard Lefschetz.

Recall that $\BC$ denotes the category of Soergel bimodules. Let
\[
K := K^b(\BC)
\]
denote its bounded homotopy category. As in \cite[\S 6.1]{EWHodge} we
denote the cohomological degree of an object by an upper left index,
so as not to get confused with the grading. Thus, an object in $K$ is a
complex
\[
\dots \to {}^iF \to {}^{i+1}F \to \dots
\]
with each ${}^i F \in \BC$. We denote by $(K^{\le 0}, K^{\ge 0})$
the perverse $t$-structure on $K$ (see \cite[\S 6.3]{EWHodge}).

  \begin{lem} \label{lem:splitpea}
Let $F = (0 \to {}^0 F \stackrel{d_0}{\to} {}^1 F \to \dots)$ be a
complex supported in non-negative homological degrees, and suppose
that $F \in K^{\ge 0}$. Then the induced map  
\[ d_0 : \pH^i({}^0 F) \to \pH^i({}^1 F) \] is split injective for all $i < 0$.
  \end{lem}

\begin{proof}

Because $F \in K^{\ge 0}$ then by definition we can find an isomorphism of complexes
\[
F \cong F_p \oplus F_c
\]
with $F_c$ contractible and $F_p$ such that $H^{i}({}^jF_p) = 0$ if $i < -j$. Only the summand $F_c$ contributes to $H^i({}^0 F)$ for $i < 0$, but the first differential in a contractible complex is a split injection.
\end{proof}

Given any $x \in W$ we denote by
\[
F_x = (\dots \to {}^{-1}F_x =0 \to {}^0F_x = B_x \to {}^1F_x \to \dots)
\]
a fixed choice of minimal complex for the Rouquier complex (unique up
to isomorphism), see \cite[\S 6.4]{EWHodge}.
The following lemma shows that tensor product with $F_x$ is left
$t$-exact.

\begin{lem} \label{lem:Kge}
For any $x \in W$, $(K^{\ge 0})F_x \subset K^{\ge 0}$ and $F_x(K^{\ge 0}) \subset K^{\ge 0}$.
\end{lem}

\begin{proof}
Because $F_x$ is a tensor product of various $F_s$, $s \in S$, it is enough to prove the lemma for $x = s$. That $(-) \otimes F_s$ preserves $K^{\ge 0}$ is proven in \cite[Lemma 6.6]{EWHodge}; the proof deduces the general statement from \cite[Lemma 6.5]{EWHodge}, which states that $B_x F_s \in K^{\ge 0}$ for all $x \in W$ and $s \in S$. The same proof shows that $F_s B_x \in K^{\ge 0}$, and consequently that $F_s \otimes (-)$ preserves $K^{\ge 0}$. \end{proof}


The following proposition is fundamental for what follows. (In rough
form it appears first in \cite{EWHodge}  as Theorem 6.9, Lemma
6.15 and Theorem 6.21.)

\begin{prop} \label{prop:dx}
  For any $x$ there exists a map
\[
d_x : B_x \to F(1)
\]
between positively polarized Soergel bimodules such that
\begin{enumerate}
\item all summands of $F$ are isomorphic to $B_z$ with $z < x$;
\item $d_x$ is isomorphic to the first differential on a Rouquier complex;
\item if $d_x^* : F \to B_x(1)$ denotes the adjoint of $d$, then
\[
d_x^* \circ d_x = B_x \rho  - (x\rho) B_x.
\]
\end{enumerate}
\end{prop}

\begin{proof}
 Except for part (2) this proposition is \cite[Proposition
 7.14]{WL}. However the reader may easily check that the inductive
 proof of \cite[Proposition 7.14]{WL} goes through if one adds the
 inductive assumption ``$d_x$ is isomorphic to the first differential
 on a Rouquier complex''. (Indeed, the proof mimics tensoring
 with a complex isomorphic to the Rouquier complex $F_s$ to carry out the induction.)
\end{proof}

Exchanging left and right actions gives:

\begin{prop} \label{prop:dy}
  For any $y$ there exists a map
\[
d_y : B_y \to G(1)
\]
between positively polarized Soergel bimodules such that:
\begin{enumerate}
\item all summands of $G$ are isomorphic to $B_z$ with $z < y$;
\item $d_y$ is isomorphic to the first differential on a Rouquier complex;
\item if $d_y^* : G \to B_y(1)$ denotes the adjoint of $d$, then
\[
d_y^* \circ d_y = \rho B_y   -  B_y(y^{-1}\rho)
\]
\end{enumerate}
\end{prop}

Putting these three statements together gives:

\begin{prop} \label{prop:wlsetup}
Consider the map
\[
 f := \left ( \begin{matrix} d_xB_y \\ B_xd_y \end{matrix} \right)  : B_xB_y \stackrel{}{\to} E(1) := FB_y(1) \oplus B_xG(1).
\]
Here, $d_x$ and $F$ are as in Proposition \ref{prop:dx}, and $d_y$ and $G$ are as in Proposition \ref{prop:dy}.
Then
\begin{enumerate}
\item the induced map
\[
f : \pH^i(B_xB_y) \to \pH^{i+1}(E)
\]
is split injective for $i < 0$;
\item if $f^*: E \to B_xB_y(1)$ denotes the adjoint of $f$ then
\[
f^* \circ f = B_x(2\rho) B_y - x(\rho) B_xB_y - B_xB_y(y^{-1}\rho).
\]
\end{enumerate}
\end{prop}

\begin{proof}The first claim follows by
noticing that $f$ is isomorphic to the first differential on a
Rouquier complex representing
\[
F_xF_y \cong (B_x \to E(1) \to \dots)(B_y
\to F(1) \to \dots).
\]
Because $F_xF_y \in K^{\ge 0}$ the first claim in the lemma follows
from Lemma \ref{lem:splitpea}.

The adjoint of $f$ is given by the matrix
\[
\left ( \begin{matrix} d_x^*B_y & 
      B_xd_y^* \end{matrix} \right)
\]
and hence
\[
f^* \circ f = (d_x^* \circ d_x) B_y + B_x(d_y^* \circ d_y) =  B_x(2\rho) B_y - x(\rho) B_xB_y -B_xB_y(y^{-1}\rho)
\]
which is the second claim in the lemma. 
\end{proof}

Similarly we have:

\begin{prop} \label{prop:wlssetup2}
Fix $a, b >0$ and consider the map
\[
 g_{a,b} := \left ( \begin{matrix} \sqrt{a} \cdot d_xB_sB_y \\
     \sqrt{b} \cdot B_xB_sd_y \end{matrix} \right)  : B_xB_sB_y \stackrel{}{\to} E(1) := FB_sB_y(1) \oplus B_xB_sG(1).
\]
Then
\begin{enumerate}
\item the induced map
\[
g_{a,b} : \pH^i(B_xB_sB_y) \to \pH^{i+1}(E)
\]
is split injective for $i < 0$;
\item if $g_{a,b}^*: E \to B_xB_y(1)$ denotes the adjoint of $g_{a,b}$ then
\[
g_{a,b}^* \circ g_{a,b} = a B_x (\rho)B_s B_y + b
B_x B_s (\rho) B_y - a(x\rho) B_xB_xB_y - b B_xB_y(y^{-1}\rho).
\]
\end{enumerate}
\end{prop}

\begin{proof}
The argument for (2) is the same as for the previous proposition.

It remains to show part (1). Note that $g_{a,b}$ is the first
differential on a complex representing
\[
F_xB_sF_y \cong (B_x \to E(1) \to \dots)B_s(B_y
\to F(1) \to \dots).
\]
and so $F_xB_sF_y \in K^{\ge 0}$ by Lemma \ref{lem:Kge}. Now (1)
follows from Lemma \ref{lem:splitpea}.
\end{proof}

The following two propositions explain the title of this section.

\begin{prop} \label{prop:weaklefschetz1}
  Fix $x, y \in W$ and suppose $RHR(x',y)$ and $RHR(x,y')$ hold for all
  $x' < x$, $y' < y$. Then $RHL(x,y)$ holds.
\end{prop}

\begin{remark} This proposition is an instance of the philosophy that
  HR in dimension $\le n-1$ implies HL in dimension $n$.
\end{remark}

\begin{proof} Let us keep the notation in the statement of Proposition
  \ref{prop:wlsetup}. We assume that $B_x B_y$ is standardly polarized and $E$ is polarized with the induced form.
  Fix $z \in W$ and consider the graded vector spaces
\[
V := H^\bullet_z(B_xB_y) \quad \text{and} \quad U := H_z^\bullet(E).
\]
These have operators $L : V^{\bullet} \to V^{\bullet + 2}$ and $L :
U^{\bullet} \to U^{\bullet + 2}$ obtained by applying $H_z^\bullet(-)$ to the maps
\begin{align*}
B_xB_y \to B_xB_y(2) & : b b' \mapsto b (\rho) b', \\
E \to E(2) &: (b b', b b') \mapsto (b \rho b', b \rho b').
\end{align*}
Also, the maps $f, f^*$ of Proposition
  \ref{prop:wlsetup} induce maps (again by taking perverse cohomology)
\[
U^\bullet \stackrel{f}{\to} V^\bullet \stackrel{f^*}{\to} U^{\bullet  +2}.
\]
These maps are morphisms of $\RM[L]$-modules. We have:
\begin{enumerate}
\item $f$ is injective in degrees $< 0$, by Proposition \ref{prop:wlsetup}(1).
\item $\langle f(v), f(v') \rangle = \langle v, f^*(f(v')) \rangle = \langle v, 2 L v' \rangle$ for all $v, v' \in V^\bullet$. The first equality holds because
  $f^*$ is the adjoint of $f$. The second equality holds by Proposition \ref{prop:wlsetup}(2), and by Lemma \ref{lem:pzero}.
\item $U$ satisfies the Hodge-Riemann bilinear relations. This is because $E$ is a positively polarized direct sum (of tensor products), and we have assumed $RHR(x',y)$ and $RHR(x,y')$, one of which applies to each direct summand of $E$.
\end{enumerate}
Now we may deduce from \cite[Lemma 2.3]{EWHodge} that $L_V^i : V^{-i}
\to V^i$ is injective and hence is an isomorphism by a comparison of
dimension. The property $HL(x,y)$ follows.
\end{proof}

\begin{prop} \label{prop:weaklefschetz2}
  Fix $x, y \in W$ and $s \in S$ and suppose $HR(x',s, y)$ and $HR(x,s,y')$ hold for all
  $x' < x$, $y' < y$. Then $HL(x,y)$ holds.
\end{prop}

\begin{proof}
  The proof is the same as that of the previous proposition, replacing Proposition \ref{prop:wlsetup} with Proposition \ref{prop:wlssetup2}.
\end{proof}

\subsection{Signs via limit arguments}
\label{subsec-signs}

In this section we will repeatedly appeal to the \emph{principle of
  conservation of signs}, which states that a continuous family of
non-degenerate symmetric forms on a real vector space
has constant signature. The following lemma, which was one of
the key techniques used by de Cataldo and Migliorini in their proof of
the Hodge-Riemann bilinear relations
in geometry \cite{dCM}, is an immediate consequence.

\begin{lem} \label{lem:conservation} Consider a polarized graded
  vector space and a continuous family of operators $L_t$ parametrized
  by a connected set. Assume all the operators in the family satisfy hard
  Lefschetz. If any member of the family 
satisfies the Hodge-Riemann bilinear relations, then they all
do. \end{lem}

To spell out this general argument in slightly more detail: one is given a finite-dimensional polarized graded vector space $V^\bullet$. A degree $2$ Lefschetz
operator induces a symmetric form on each $V^{-i}$, $i \in \ZM_{\ge 0}$, which collectively are non-degenerate if and only if $L$ satisfies hard Lefschetz. If $L$ does satisfy hard
Lefschetz, then $L$ satisfies the Hodge-Riemann bilinear relations if
and only if the signature of the Lefschetz form on each $V^{-i}$
agrees with a certain formula, which depends only on the
graded dimension of $V$. From this, one deduces the lemma above. The
applications will become clear immediately.

\begin{prop} \label{prop:xsbigger}
Suppose $x, y \in W$, $s \in S$ and $xs > x$. Assume $RHL(x,s,y)$ and  $RHR(\le xs,y)$. Then $RHR(x,s,y)$ holds.
\end{prop}

\begin{proof} For $a, b \in \RM$, consider the Lefschetz operator
\[ L_{a,b} := B_x(a\rho) B_sB_y + B_xB_s(b\rho) B_y : B_xB_sB_y \to B_xB_sB_y(2).\]
Recall that $HR(x,s,y)$ means that $L_{a,b}$ induces an operator on
$H_z^\bullet(B_xB_sB_y)$ which satisfies hard Lefschetz and
Hodge-Riemann, for any $a > 0$, $b > 0$.

However $B_xB_s$ is perverse, and by $RHR(x,s)$ (see Lemma \ref{ys} above) the restriction of the intersection form on $B_x B_s$ to each summand $B_z \sumset B_x B_s$ is a multiple of the
intersection form on $B_z$ with sign $(-1)^{(\ell(x) + 1 -
  \ell(z))/2}$. By $RHR(\le xs, y)$, $L_{0, b}$ satisfies relative Hodge-Riemann on $B_xB_sB_y$ for any $b > 0$ (it is an exercise
to confirm that the signs are correct). 
Thus $L_{a,b}$ satisfies relative hard Lefschetz for all $a \ge 0$ and
$b > 0$ and satisfies relative Hodge-Riemann for $a = 0$, $b > 0$.
 We can now appeal to the principle of conservation of signs to
 conclude that relative Hodge-Riemann is satisfied for all $a \ge 0, b
 > 0$. Thus $RHR(x,s,y)$ holds. \end{proof}

The previous proof uses the case special case $a = 0, b > 0$ to deduce
the general case $a > 0, b > 0$. Here we go the other way:

\begin{prop} \label{prop:xy}
Suppose $x, y \in W$, $s \in S$ and that $sy > y$. Assume $RHR(x,s,y)$
and $RHL(x, \le sy)$. Then $RHR(x,sy)$ holds.
\end{prop}

\begin{proof}
  Let $L_{a,b}$ denote the Lefschetz operator considered in the
  previous proof. By our assumptions, $L_{a,b}$ satisfies Hodge-Riemann for $a >
  0, b > 0$ and hard Lefschetz for $a > 0, b = 0$. By the principle of
  conservation of signs, Hodge-Riemann is also satisfied for $a > 0, b
  = 0$. Now $B_xB_{sy}$ is a summand of $B_xB_sB_y$ and
  the intersection form on $B_xB_sB_y$ restricts to a positive
  multiple of the intersection form on $B_xB_{sy}$. We
  conclude\footnote{We are using the fact that relative Hodge-Riemann
    is preserved under taking polarized direct summands. See
    \cite[Lemma 4.5]{WL} for a related situation.}  that
  $L_{a, 0}$ satisfies Hodge-Riemann on $B_xB_{sy}$, which is what we
  wanted.
\end{proof}

\begin{prop} \label{prop:xssmaller}
Let $x, y \in W$ and $s \in S$ be such that $xs < x$. Assume $HL(x,s,y)$, $HR(x,y)$. Then $HR(x,s,y)$ holds.
\end{prop}

The proof of Proposition \ref{prop:xssmaller} is more complicated than
that of Proposition \ref{prop:xsbigger}, and will occupy the rest of
this section. Here is
a sketch of our approach. We fix a decomposition $B_xB_s = B_x(1)
\oplus B_x(-1)$ and explicitly calculate the Lefschetz operator and forms in the
decomposition 
\[
B_xB_sB_y = B_xB_y(1) \oplus B_xB_y(-1)
\]
in terms of the corresponding operators on
$B_xB_y$. Appealing to $RHR(x,y)$ we will see that the signs are correct
for $b \gg a > 0$. By the principle of conservation
of signs (which is applicable by our $RHL(x,s,y)$ assumption) we
deduce that $RHR(x,s,y)$ holds, which is what we wanted to show.

For simplicity we assume $\rho(\alpha_s^\vee) = 1$ for all $s \in S$.

\begin{lem}
  The map $r \mapsto ( \partial_s(-rs(\rho)), \rho \partial_s(r))$
  gives an isomorphism
\begin{equation} \label{Rsplit}
R = R^s \oplus \rho R^s
\end{equation}
of $R^s$-bimodules.
\end{lem}

\begin{proof}
  $R$ is free as an $R^s$-module with basis $\{ 1, \g \}$ where $\g \in
  R^2$ is any degree two element which is not $s$-invariant. In particular we
  can take $\g = \rho$. Under the map as in the statement of the lemma we have 
  \begin{gather*}
    1 \mapsto (\partial_s(-s\rho), \rho \partial_s(1)) = (1, 0) \\
    \rho \mapsto (\partial_s( - \rho s(\rho)), \rho \partial_s(\rho) =
    (0, \rho).
  \end{gather*}
and so our map sends a basis to a basis, and the lemma follows.
\end{proof}

By \cite[Proposition 7.4.3]{Will} there exists a $(R,R^s)$-bimodule
$B_x^s$ (a ``singular Soergel bimodule'') and a canonical isomorphism
\begin{equation}
B^s_x \otimes_{R^s} R = B_x.\label{eq:7}
\end{equation}

Our choice of isomorphism \eqref{Rsplit} yields a decomposition
\begin{gather}
  \label{eq:4}
  B_xB_s = B^s _x \otimes_{R^s} R \otimes_{R^s} R(1) = B_x(1)
  \oplus B_x(-1).
\end{gather}
Now consider the endomorphism $B_x\rho B_s  : B_xB_s \to
B_xB_s(2)$.

\begin{lem} \label{lem:operator}
  With respect to the decomposition \eqref{eq:4} the degree 2
  endomorphism $B_x\rho B_s$ is given by the matrix:
  \begin{equation}
    \label{eq:8}
 \left (   \begin{matrix} 0 & B_x (-\rho(s\rho))  \\ B_x & B_x(\rho +
     s\rho)  \end{matrix} \right ) : B_x(1) \oplus B_x(-1) \to B_x(3) \oplus B_x(1).
  \end{equation}
\end{lem}

\begin{proof}
We identify $B_x$ with $B_x^s \ot_{R^s} R$, and write an element of it as $b \ot f$ for $b \in B_x^s$ and $f \in R$. Similarly, we identify $B_x B_s$ with $B_x^s \ot_{R^s} R \ot_{R^s} R(1)$.

Consider an element of the form $b \otimes 1 \in B_x$. We calculate the action of $B_x \rho B_s$ on the summand $B_x(1)$:
\begin{equation*}
    \begin{tikzpicture}[xscale=3,yscale=0.8]
\node (z0) at (0,1) {$B_x(1)$};      
\node (z1) at (1,1) {$B_xB_s$};
\node (z2) at (2,1) {$B_xB_s$};
\node (z3) at (3,1) {$B_x(1) \oplus B_x(-1)$};
\node (x0) at (0,0) {$b \otimes 1$};      
\node (x1) at (1,0) {$b \otimes 1 \otimes 1$};
\node (x2) at (2,0) {$b \otimes \rho \otimes 1$};
\node (x3) at (3,0) {$(0, b \otimes 1)$};
\draw[right hook-latex] (z0) to node[above] {\tiny \eqref{eq:7}} (z1);
\draw[->] (z1) to node[above] {\tiny $B_x\rho B_s$} (z2);
\draw[->] (z2) to node[above] {\tiny \eqref{eq:7}} (z3);
\draw[|->] (x0) to (x1);
\draw[|->] (x1) to (x2);
\draw[|->] (x2) to (x3);
    \end{tikzpicture}
  \end{equation*}
Similarly we calculate the action on the summand $B_x(-1)$:
  \begin{equation*}
    \begin{tikzpicture}[xscale=2.6,yscale=0.8]
\node (z0) at (0,1) {$B_x(-1)$};      
\node (z1) at (1,1) {$B_xB_s$};
\node (z2) at (2,1) {$B_xB_s$};
\node (z3) at (3.5,1) {$B_x(1) \oplus B_x(-1)$};
\node (x0) at (0,0) {$b \otimes 1$};      
\node (x1) at (1,0) {$b \otimes \rho \otimes 1$};
\node (x2) at (2,0) {$b \otimes \rho^2 \otimes 1$};
\node (x3) at (3.5,0) {$(b \otimes (-\rho s(\rho)), 
  b \otimes (\rho + s\rho))$};
\draw[right hook-latex] (z0) to node[above] {\tiny \eqref{eq:4}} (z1);
\draw[->] (z1) to node[above] {\tiny $B_x \rho B_s$} (z2);
\draw[->] (z2) to node[above] {\tiny \eqref{eq:4}} (z3);
\draw[|->] (x0) to (x1);
\draw[|->] (x1) to (x2);
\draw[|->] (x2) to (x3);
    \end{tikzpicture}
  \end{equation*}
The lemma follows.
\end{proof}

\begin{lem}
  The singular Soergel bimodule $B_x^s$ admits a unique invariant form
\[
\langle -, -\rangle_{B_x^s} : B_x^s \times B_x^s \to R^s
\]
such that $\langle -, -\rangle \otimes_{R^s} R$ agrees with the
intersection form under the identification \eqref{eq:7}.
\end{lem}

Here and in the following proof, an invariant form on an $(R,R^s)$-bimodule means a graded
bilinear form $\langle - , -\rangle: B_x^s \times B_x^s \to R^s$ which
satisfies $\langle rb, b' \rangle = \langle b, rb' \rangle$ and
$\langle br', b' \rangle = \langle b, b'r' \rangle = \langle b, b'
\rangle r'$ for all $b, b' \in B_x^s$, $r \in R$, $r' \in R^s$.

\begin{proof} 
Let ${}^sB_{x^{-1}}$ denote the $(R^s, R)$-bimodule obtained from
$B^s_x$ by interchanging left and right actions. Then
${}^sB_{x^{-1}}$ agrees with the indecomposable singular Soergel
bimodule parametrized by the coset of $x^{-1}$ in $\langle s \rangle
\setminus W$, as described in \cite[Theorem 7.4.2]{Will}. Soergel's
conjecture and \cite[Theorem 7.4.1]{Will} implies that
$\Hom({}^sB_{x^{-1}}, \DM({}^sB_{x^{-1}}))$ is one dimensional. (We
denote by $\DM$ the duality functor on singular Soergel bimodules
defined in \cite[\S 6.3]{Will}.) We can regard elements in this Hom
space as maps $B^s_x \to \Hom_{-R^I}(B_x^s, R^I)$ and hence as
invariant forms
\[
\langle -, - \rangle : B_x^s \times B_x^s \to R^I.
\]
We conclude that $B_x^s$ admits an invariant form which is unique
up to scalar. Given any such form $\langle -, -\rangle$, $\langle -,
-\rangle \otimes_{R^s} R$ is a non-degenerate form on $B_x$, and hence agrees with the
intersection form on $B_x$ up to scalar. The lemma follows.
\end{proof}

Our fixed decomposition \eqref{eq:4} gives the basic identification:
\begin{equation}
  \label{eq:basic}
  B_xB_sB_y = B_xB_y(1) \oplus B_xB_y(-1)
\end{equation}
The following is immediate from the definitions:

\begin{lem} \label{lem:form}
  Under \eqref{eq:basic} the
  invariant form is given by:
\[
\langle (b_1, b_2) , (b_1', b_2') \rangle = \langle b_1, b_2' \rangle
+ \langle b_2, b_1' \rangle + \langle \rho b_2, b_2' \rangle.
\]
\end{lem}

We now put the above calculations together. Until the end of the section let us in addition fix $z \in W$ and set 
\[
V^\bullet := H_z^\bullet(B_xB_y).
\]
Then $V^\bullet$ is equipped with a symmetric form $\langle -,
-\rangle_{V^\bullet}$ and a Lefschetz operator $L : V^\bullet \to
V^{\bullet + 2}$. This data satisfies Hodge-Riemann, by our assumption
$HR(x,y)$. Our identification \eqref{eq:basic} fixes an isomorphism
\begin{equation}
  \label{eq:basicz}
H_z^\bullet(B_xB_sB_y) = V^\bullet(1) \oplus V^\bullet(-1).
\end{equation}

\begin{prop} Under the identification \eqref{eq:basicz}:
  \begin{enumerate}
  \item The invariant form is given by
\begin{equation} \label{eq:form}
\langle (v_1, v_2), (v_1', v_2') \rangle = \langle v_1, v_2' \rangle +
\langle v_2, v_1' \rangle + \langle v_2, L v_2 \rangle.
\end{equation}
for $v_1, v_1' \in V^\bullet(1)$ and $v_2, v_2' \in V^\bullet(-1)$. 
\item The operator induced by $L_{a,b} := B_x(a\rho) B_sB_y +
  B_xB_s(b\rho)B_y$ is given by
  \begin{equation}
    \label{eq:splitting}
    a \left ( \begin{matrix} 0 & X \\ \id & Y \end{matrix} \right ) +
b \left ( \begin{matrix} L & 0 \\ 0 & L \end{matrix} \right )
  \end{equation}
for certain (unspecified) maps $X : V(-1) \to V(1)$ and $Y : V(-1) \to V(-1)$.
  \end{enumerate}
\end{prop}

\begin{proof}
(1) (resp. (2)) is an immediate consequence of Lemma \ref{lem:form}
(resp. Lemma \ref{lem:operator}).
\end{proof}

\begin{prop}
  Assume $HR(x,y)$. Then for $b \gg a > 0$ the operator
  $L_{a,b}$ satisfies HR on $V^\bullet(1) \oplus V^\bullet(-1)$.
\end{prop}

\begin{proof}
  We roll up our sleeves and calculate everything in a basis.

Fix a degree $-d \le 0$. By \cite[Lemma 5.2]{EWHodge} it is enough to
show that for $b \gg a > 0$ the signature of the Lefschetz form on the
degree $-d$ piece of $V^\bullet(1) \oplus V^\bullet(-1)$ is equal
to the signature of the Lefschetz form on the primitive subspace
\[
P^{-d+1} := \ker L^{d} : V^{-d+1} \to V^{d+1}.
\]
To this end let us fix bases:
  \begin{gather*}
    x_1, \dots, x_m \quad \text{for $V^{-d-1}$}; \\
    p_1, \dots, p_n \quad \text{for $P^{-d+1}$}.
  \end{gather*}
Because $L$ satisfies hard Lefschetz on $V$ we deduce that
  \begin{gather*}
    Lx_1, \dots, Lx_m, p_1, \dots, p_n \quad \text{is a basis for $V^{-d+1}$}.
  \end{gather*}
Thus a basis for $(V^\bullet(1) \oplus V^\bullet(-1))^d = V^{d+1}
\oplus V^{d-1}$ is given by
  \begin{gather*}
(0,x_1), \dots, (0, x_m), (Lx_1,0), \dots, (Lx_m,0), (p_1,0), \dots, (p_n,0).
\end{gather*}

Let us write
\[
L_{a,b} = aA + bB
\]
where 
\[ 
A =  \left ( \begin{matrix} 0 & X \\ \id & Y \end{matrix} \right
    ) \quad \text{and} \quad B = \left ( \begin{matrix} L & 0 \\ 0 &
        L \end{matrix} \right ) 
\]
are the matrices appearing in Proposition \ref{eq:splitting}. We
calculate the leading terms of the Lefschetz form $(v, w) \mapsto
\langle v, 
L_{a,b}^{d}w \rangle$ in the above basis with respect to the parameter
$b$. We have:
\begin{gather*}
  \langle (0, x_i), L_{a,b}^d(0,x_j) \rangle = b^d\langle Lx_i, L^dx_j
  \rangle_{V^\bullet} + O(b^{d-1}) = b^d\langle x_i, L^{d+1}x_j \rangle_{V^\bullet} + O(b^{d-1}) \\
  \langle (Lx_i, 0), L_{a,b}^d(0,x_j) \rangle = b^d \langle Lx_i,
  L^dx_j \rangle_{V^\bullet}  + O(b^{d-1}) = b^d \langle x_i,
  L^{d+1}x_j \rangle_{V^\bullet}  + O(b^{d-1}) \\
\langle (Lx_i, 0), L_{a,b}^d(Lx_i,0) \rangle = b^d \langle (Lx_i, 0),
(L^{d+1}x_i,0) \rangle + O(b^{d-1}) = O(b^{d-1})
\end{gather*}
where $O(b^{k})$ denotes a polynomial in $b$ and $a$ in which all
powers of $b$ are bounded by $k$. Using that $L^dp_i = 0$ we have
\begin{gather*}
\langle (0, x_i), L_{a,b}^d(p_i,0) \rangle = d ab^{d-1} \langle 
L x_i, L^{d-1}p_i \rangle + O(b^{d-2}) = O(b^{d-2}) \\
\langle (Lx_i, 0), L_{a,b}^d(p_i,0) \rangle = d ab^{d-1} \langle 
Lx_i, L^{d-1}p_i \rangle + O(b^{d-2}) = O(b^{d-2}) \\
\langle (p_i, 0), L_{a,b}^d(p_i,0) \rangle = d ab^{d-1} \langle 
p_i, L^{d-1}p_i \rangle + O(b^{d-2}).
\end{gather*}
Thus if we define matrices
\[
R := ( \langle x_i, L^d x_j \rangle )_{1 \le i, j \le m} \; \quad
\text{and} \quad
Q := ( \langle p_i, L^d p_j \rangle )_{1 \le i, j \le n}
\]
then we can write the Gram matrix of the Lefschetz form 
$(v, w) \mapsto \langle v, L_{a,b}^{d}w \rangle$ as a block
matrix with entries:
\begin{gather*}
 \left (    \begin{matrix}
  b^d R + O(b^{d-1}) & b^d R + O(b^{d-1}) & O(b^{d-2}) \\
    b^d R + O(b^{d-1}) & O(b^{d-1}) & O(b^{d-2}) \\
   O(b^{d-2}) & O(b^{d-2})& dab^{d-1}Q + O(b^{d-2})
  \end{matrix} \right )
\end{gather*}
For $b \gg a > 0$ this matrix has the same signature of the matrix
\begin{gather*}
 \left (    \begin{matrix}
  R & R& 0 \\
   R & 0 & 0 \\
   0 & 0 & Q
  \end{matrix} \right )
\end{gather*}
Now the submatrix $\left (    \begin{matrix}
  R & R \\
   R & 0
  \end{matrix} \right )$ is easily seen to be non-degenerate with signature 0. Thus for
$b \gg a > 0$ our matrix has the same signature of $Q$. We have
already remarked that by \cite[Lemma 5.2]{EWHodge} this is what we
wanted to know.
\end{proof}

Thus Proposition \ref{prop:xssmaller} holds (see the remarks
immediately after the statement of the proposition).

\section{Ridigity}

Let $\cell \subset W$ be a two-sided cell, $a$ its
$a$-value\footnote{It is a non-trivial fact (a consequence of
  Soergel's conjecture) that the $a$-function is well-defined for any Coxeter
  group, see \cite[\S 10.1]{LusztigTC}.}
 and $J
=\bigoplus_{x \in \cell} \ZM j_x$ the
$J$-ring associated to $\cell$ ($J$ is denoted $J^\cell$ in \cite[\S~18.3]{LusztigUP}).
Following Lusztig \cite[\S~10]{LusztigTC}, we define a semi-simple monoidal
category $\JC$ ($\JC$ is denoted $C^\cell$ in
\cite[\S~18.5]{LusztigUP}).

We first consider the subcategory $\BC_{< \cell} \subset \BC$, consisting of all direct sums of shifts of $B_z$ with $z <_{LR}
\cell$ ($<_{LR}$ denotes the two-sided preorder). Let $I_\cell$ denote the ideal in $\BC$ consisting of all morphisms which factor through objects in $\BC_{< \cell}$. Because $\BC_{< \cell}$ is closed under tensor products with arbitrary objects of $\BC$, $I_\cell$ is a
tensor ideal in $\BC$, and we can form the quotient of additive
categories $\BC_{\cell}' := \BC / I_\cell$. Then $\BC_\cell'$ is a
graded additive monoidal category and we set $\BC_\cell$ to be the full graded
additive subcategory generated by $B_x$ with $x \in \cell$. We denote
the image of $B_x$ in $\BC_\cell'$ by $B_x^{\cell}$. 
The objects $B_x^{\cell}(m)$ with $x \in W$ and $x \not < \cell$
(resp. $x \in \cell$) give representatives for the isomorphism classes
of the indecomposable objects in $\BC_\cell'$
(resp. $\BC_\cell$). Moreover $\BC_\cell$ is a graded additive monoidal category (without unit
unless $\cell = \{ \id \}$).

The (obvious analogues of the) crucial
vanishing statements \eqref{eq:deltaxy} and \eqref{eq:perv_vanish}
still hold in $\BC'_\cell$ and $\BC_\cell$, and hence the perverse
filtration and perverse cohomology functors descend to $\BC_\cell'$
and $\BC_\cell$. We denote them by the same symbols. It is immediate
from the definition of the $a$-function that, for all $x, y \in \cell$,
\begin{equation}
  \label{eq:avanishing}
H^i(B_x^{\cell}B_y^{\cell}) = 0 \quad \text{if $|i| > a$.}  
\end{equation}

We now come to the definition of $\JC$. It is a full subcategory of $\BC_\cell$, although with a different monoidal structure. The objects of $\JC$ are given
by direct sums (without shifts) of $B_x^{\cell}$ with $x \in
\cell$, and thus by \eqref{eq:deltaxy} the category is semi-simple. The
monoidal product is given by
\[
B * B' := H^{-a}(BB') \in \BC_\cell
\]
(the lowest potentially non-zero degree, by \eqref{eq:avanishing}).
Lusztig proves that $\JC$ is a semi-simple monoidal category (this
result relies in an essential way on \cite{EWHodge}), and that
the map $j_x \to [B_x^{\cell}]$ induces an isomorphism
$J \simto [\JC]$, where $[\JC]$ denotes the Grothendieck group of
$\JC$.

\begin{remark} \label{rem:unit} The reader is warned that in general $\JC$ is a
  ``monoidal category without unit'', i.e. it has an associator but no
  unit. In general, Lusztig conjectures \cite[\S 13.4]{LusztigUP} that
  the $a$-function is bounded (i.e. $a(z) \le N$ for all $z \in
  W$ and some fixed constant $N$, which he describes explicitly). This
  boundedness is known to hold for finite and affine Weyl 
  groups. 
Under the assumption of this conjecture, it turns out that $\JC$ has a unit if and only if $\cell$
  contains finitely many left cells (as is always the case in finite
  and affine
  type). In this case Lusztig proves \cite[\S
  18.5]{LusztigUP} that the object
  $\bigoplus_{x \in \DC \cap \cell} B_x^{\cell}$ is a unit for $\JC$ (here
  $\DC \subset W$ denotes the set of distinguished involutions).
  Even when $\cell$ contains infinitely many left cells
  $\JC$ is ``locally unital'' (and still under the boundedness assumption). For any given object $B \in \JC$, only
  finitely many $B_x^{\cell}$ with $x \in \DC \cap \cell$ satisfy
  $B_x^\cell * B \neq 0$. The formal direct sum $\bigoplus_{x \in \DC
    \cap \cell} B_x^{\cell}$, while not an object in $\JC$ when $\DC
  \cap \cell$ is infinite, acts on any object, and it will act as a
  monoidal identity would.
\end{remark}

Our aim in this chapter is to show that the relative hard Lefschetz
theorem for Soergel bimodules implies

\begin{thm} \label{thm:rigid}
  $\JC$ is a rigid, pivotal monoidal category.
\end{thm}

\begin{remark}
  For finite and affine Weyl groups the rigidity of 
  $\JC$ has been proved by Bezrukavnikov, Finkelberg and Ostrik 
  \cite[\S 4.3]{BFO1} (using the geometric Satake equivalence). Lusztig has also proven rigidity for Weyl groups
  (see \cite[\S 9.3]{LusztigTC} and
  \cite[\S 18.19]{LusztigUP}). His techniques probably extend to
  crystallographic Coxeter groups. Lusztig also conjectured the rigidity to
  hold for any finite Coxeter group \cite[\S 10]{LusztigTC}, in which
  case he expects the Drinfeld center $Z(\JC)$ to be related to the
  ``unipotent characters'' of $W$. Ostrik
  has informed us that for the interesting case of the two-sided cell
  in $H_4$ with $a$-value 6, he has been able to verify
  the rigidity of $\JC$ by other means.
\end{remark}

\begin{remark}
  As we will see, the pivotal structure on $\JC$ will depend on our
  fixed choice of regular dominant element $\rho \in \hg^*$. We do not
  know if the structure varies in an interesting way with $\rho$. It
  is possible that the Hodge-Riemann relations might allow one to
  show that $\JC$ is unitary, and hope to address this question in
  future work.
\end{remark}

Because $\JC$ does not have a unit in general the standard definition
of rigidity does not make sense. We will prove the
following (which is equivalent to the usual notion of rigidity if
$\JC$ has a unit, see Remark \ref{rem:rig} below):

\begin{prop} \label{prop:rigid}
  \begin{enumerate}
  \item For $B, X, Y \in \JC$ we have canonical isomorphisms
    \begin{align*}
      \Hom_\JC(X, B*Y) \stackrel{\phi_{X,Y}}{\longto} \Hom_\JC(B^\vee * X, Y) \\
      \Hom_\JC(X, Y*B) \stackrel{\chi_{X,Y}}{\longto} \Hom_\JC(X * B^\vee, Y)
    \end{align*}
functorial in $X$ and $Y$.
\item For $B, X, Y, Z \in \JC$ the following diagrams commute:
\begin{align} \label{YZ}
\begin{array}{c}
    \begin{tikzpicture}[xscale=5,yscale=1.2]
\node (z0) at (0,1) {$\Hom_{\JC}(X, B*Y)$};
\node (z1) at (0,0) {$\Hom_{\JC}(B^\vee*X, Y)$};
\node (z2) at (1,1) {$\Hom_{\JC}(X*Z, B*Y*Z)$};
\node (z3) at (1,0) {$\Hom_{\JC}(B^\vee*X*Z, Y*Z)$};
\draw[->] (z0) to node[right] {\tiny $\phi_{X,Y}$} (z1);
\draw[->] (z2) to node[right] {\tiny $\phi_{X*Z,Y*Z}$} (z3);
\draw[->] (z0) to node[above] {\tiny $(-)*Z$} (z2);
\draw[->] (z1) to node[above] {\tiny $(-)*Z$} (z3);
    \end{tikzpicture} \end{array}
\\
\label{ZY}
\begin{array}{c}
    \begin{tikzpicture}[xscale=5,yscale=1.2]
\node (z0) at (0,1) {$\Hom_{\JC}(X, Y*B)$};
\node (z1) at (0,0) {$\Hom_{\JC}(X*B^\vee, Y)$};
\node (z2) at (1,1) {$\Hom_{\JC}(Z*X, Z*Y*B)$};
\node (z3) at (1,0) {$\Hom_{\JC}(Z*X*B^\vee, Z*Y)$};
\draw[->] (z0) to node[right] {\tiny $\chi_{X,Y}$} (z1);
\draw[->] (z2) to node[right] {\tiny $\chi_{Z*X,Z*Y}$} (z3);
\draw[->] (z0) to node[above] {\tiny $Z*(-)$} (z2);
\draw[->] (z1) to node[above] {\tiny $Z*(-)$} (z3);
    \end{tikzpicture}
\end{array}
\end{align}
  \end{enumerate}
\end{prop}

We  make some remarks before turning to the proof. It is easy to
see that $B_s \in \BC$ is self-dual (this is immediate in the language of
\cite{EWSC}, where the cup and cap maps provide the unit and
counit). It follows that any Bott-Samelson module is rigid. Hence $\BC$ is
rigid (taking the Karoubi envelope preserves 
rigidity).  Let us denote by $B \mapsto B^\vee$ the
duality on $\BC$. It is easy to see that $B$ is even pivotal
(i.e. we have a canonical isomorphism $B \simto (B^\vee)^{\vee}$).

As quotients of a rigid, pivotal monoidal category, the monoidal
categories $\BC_\cell$ and $\BC_\cell'$ are rigid and pivotal. We abuse notation and also denote
the duality on $\BC_\cell$ by $B \mapsto B^\vee$.

\begin{proof} We first establish (1). We will construct the isomorphism
  $\phi_{X,Y}$, the proof for $\chi_{X,Y}$ is similar. Let $X, Y, B
  \in \JC$. We have canonical identifications (by definition and
  the analogue for $\BC_\cell$ of
  \eqref{eq:perv_vanish2})
  \begin{align*}
\Hom_{\JC}(X, B*Y) & = \Hom_{\BC_\cell}(X, H^{-a}(BY)) =
\Hom_{\BC_\cell}(X, BY(-a)) = \\
& = \Hom_{\BC_\cell}(B^\vee X, Y(-a)) = \Hom_{\BC_\cell}(H^{a}(B^\vee X), Y).
  \end{align*}
Precomposing with the isomorphism $H^{-a}(B^\vee X) \simto
H^{a}(B^\vee X)$ given by relative hard Lefschetz gives an
isomorphism
\[
\Hom_{\BC_\cell}(H^{a}(B^\vee X), Y)  \simto
\Hom_{\BC_\cell}(H^{-a}(B^\vee X), Y)   = \Hom_{\JC}( B^\vee * X, Y).
\]
The composition of these isomorphisms defines our isomorphism
$\phi_{X,Y}$. It is immediate to check that this isomorphism is
natural in $X$ and $Y$.

We now turn to (2). As before we only establish the commutativity of
\eqref{YZ}, with \eqref{ZY} being similar. Choose $f \in \Hom_{\JC}(X, B*Y)$ and let $f_{NE}$ (resp. $f_{SW}$)
denote the image of $f$ in $\Hom_{\JC}(B^\vee * X * Z, Y * Z)$ obtained
by passing  through the north-east (resp. south-west) corner of
\eqref{YZ}. We must prove that $f_{SW} = f_{NE}$.

Via $\Hom_{\JC}(X, BY) = \Hom_{\BC_\cell}(X, BY(-a))$ we may regard
$f$ as a map
\[
f : X \to BY(-a).
\]
From $f$ we  obtain the following maps:
\begin{align*}
  f' : B^\vee X \to Y(-a), & \quad \varphi : H^a(f') : H^a(B^\vee X)  \to Y, \\
g := fZ : XZ \to BYZ(-a), & \\
g' := f'Z : B^\vee XZ \to YZ(-a), & \quad \g := H^{2a}(g') : H^{2a}(B^\vee X
Z) \to H^a(YZ) \\
h : H^{-a}(XZ) \to B H^{-a}(YZ)(-a), & \quad h' : B^\vee H^{-a}(XZ) \to H^{-a}(YZ)(-a)
\end{align*}
(Here $f'$ (resp. $g'$, $h'$) are obtained from $f$ (resp. $g$, $h$)
using the dual pair $(B, B^\vee)$ in $\BC_\cell$, and $h$ is uniquely
determined by $H^0(h) = H^{-a}(g)$.)

Consider the diagram given in Figure \ref{fig:rigid}. The maps which
have not been defined above are given as follows:
\begin{enumerate}
\item All maps labelled $\sim$ are relative hard Lefschetz
  isomorphisms (given by our fixed choice of $\rho \in \hg^*$). At the
  top and bottom of the middle square we use the canonical identifications:
  \begin{gather*}
    H^{2a}(B^\vee X Z) = H^{a}( H^{a}(B^\vee X) Z) = H^{a}( B^\vee
    H^{a}(XZ)), \\
    H^{-2a}(B^\vee X Z) = H^{-a}( H^{-a}(B^\vee X) Z) = H^{-a}( B^\vee H^{-a}(XZ))
  \end{gather*}
\item We set $l := H^{-a}(\varphi Z)$ and $r := H^a(h')$.
\end{enumerate}

\begin{figure}[h] 
\caption{Diagram for the proof of Proposition \ref{prop:rigid}(2).}
\label{fig:rigid}
\[
\begin{array}{c}
    \begin{tikzpicture}[xscale=2.8,yscale=2]
\node (A) at (1,0) {{\small $H^{a}(B^\vee H^{-a}(X Z))$}};
\node (t) at (0,-1) {{\small $H^{-2a}(B^\vee X Z)$}};
\node (a) at (-1,0) {{\small $H^{-a}(H^{a}(B^\vee X) Z)$}};
\node (T) at (0,1) {{\small $H^{2a}(B^\vee X Z)$}};
\node (B) at (2,0) {{\small $H^{-a}(YZ)$}};
\node (b) at (-2,0) {{\small $H^{-a}(YZ)$}};
\node (C) at (2,1) {{\small $H^{a}(YZ)$}};
\node (c) at (-2,1) {{\small $H^{a}(YZ)$}};
\draw[->] (t) to node[below] {\small $\sim$} (A);
\draw[->] (t) to node[below] {\small $\sim$} (a);
\draw[->] (A) to node[above] {\small $\sim$} (T);
\draw[->] (a) to node[above] {\small $\sim$} (T);
\draw[->] (A) to node[above] {\small $r$} (B);
\draw[->] (a) to node[above] {\small $l$} (b);
\draw[->] (b) to node[left] {\small $\sim$} (c);
\draw[->] (B) to node[right] {\small $\sim$} (C);
\draw[->] (T) to node[above] {\small $\g$} (C);
\draw[->] (T) to node[above] {\small $\g$} (c);
\draw[->] (t) to[out=180,in=-90] node[below] {\small $f_{SW}$} (b);
\draw[->] (t) to[out=0,in=-90] node[below] {\small $f_{NE}$} (B);
    \end{tikzpicture}
\end{array}
\]
\end{figure}

It is straightforward but tedious to check that all squares and triangles in Figure
\ref{fig:rigid} commute. If $q$ denotes the relative hard Lefschetz isomorphism
$q : H^{-a}(YZ) \to H^a(YZ)$ we deduce from the commutativity of the
 diagram that $q  \circ f_{NE} = q  \circ
f_{SW}$, and hence that $f_{NE} = f_{SW}$, which is what we wanted to show.
\end{proof}


\begin{remark}
 \label{rem:rig} 
Suppose that $\cell$ contains finitely many left cells. Then  $\JC$
has a unit (see Remark \ref{rem:unit}), which we denote by
$\mathbb{1}$. Applying the isomorphisms of Proposition
\ref{prop:rigid} to the identity maps in $\Hom_{\JC}(B,B)$ and $\Hom_{\JC}(B^\vee, B^\vee)$,
we obtain morphisms $\e : \mathbb{1} \to B *B^\vee$ and $\mu : B^\vee*B \to
\mathbb{1}$. Using the naturality of Proposition \ref{prop:rigid}(1) and
the commutativity of Proposition \ref{prop:rigid}(2) one may check
that for $f : X \to B*Y$, $\phi_{X,Y}(f)$ is given by the
composition\footnote{Here are more details: by Proposition \ref{prop:rigid}(2) one can show that $\phi_{B * Y,Y}(\id_{B * Y}) = \mu * Y$. By naturality of Proposition \ref{prop:rigid}(1) under precomposition with $f$, one obtains the desired equality.}
\[
B^\vee * X \stackrel{B^\vee * f}{\longto} B^\vee * B * Y
\stackrel{ \mu * Y}{\longto} \mathbb{1}*Y = Y.
\]
Similarly, the inverse of $\phi_{X,Y}$ sends $g : B^\vee*X \to Y$ to
\[
X = \mathbb{1}* X \stackrel{\e * X}{\longto} B * B^\vee * X
\stackrel{B * g}{\longto} B * Y.
\]
From this one easily deduces that $B^\vee$ (and $\e, \mu$) is left
dual to $B$. Similarly, one deduces that $B^\vee$ is right dual to
$B$. Hence $\JC$ is rigid in the usual sense.
\end{remark}


\end{document}